\documentclass[a4paper,leqno]{article}
\usepackage[utf8]{inputenc}
\usepackage{amsmath}
\usepackage{amssymb}
\usepackage{graphicx}
\usepackage[all,dvips]{xy}
\usepackage{amsthm}
\usepackage[english]{babel}
\usepackage{ifpdf}
\usepackage{color}
\usepackage[active]{srcltx}
\usepackage{lmodern}
\usepackage{enumerate}
\usepackage{fancybox}
\usepackage{amsbsy}
\usepackage{pstricks}

\def\eps{\varepsilon}

\def\sig{\sigma}

\def\Om{\Omega}

\def\Ga{\Gamma}

\def\tt{\theta}

\def\R{{\mathbb R}}

\def\P{{\mathbb P}}
\def\C{{\mathcal C}} 

\def\T{\mathcal T}

\def\div{{\rm div}}

\def\sup{{\rm sup}}
\def\max{\rm max}

\def\nuu{\overrightarrow{\nu}}
\def\Gn{G\cdot\nu}

\def\W{W^{1,1}((0,+\infty)\, ;\, L^1(\Ga))}
\def\Winf{W^{1,\infty}}

\def\trho{\tilda{\rho}}

\def\tQ{\tilda{Q}}
\def\tB{\tilda{B}}

\def\tphi{\tilda{\phi}}

\def\bOm{\barre{\Om}}

\newtheorem{thr}{Theorem}
\newtheorem{defn}{Definition}
\newtheorem{prop}{Proposition}
\newtheorem{lem}{Lemma}
\newtheorem{rem}{Remark}
\newtheorem{cor}{Corollary}

\numberwithin{equation}{section}

\def\div{{\rm div}}

\def\d{\partial}

\newcommand{\ncmd}{\newcommand}
\ncmd{\tilda}{\widetilde}
\ncmd{\abs}[1]{\left| #1 \right|}
\ncmd{\beq}{\begin{equation}} \ncmd{\eeq}{\end{equation}}
\ncmd{\beqn}{\begin{equation*}} \ncmd{\eeqn}{\end{equation*}}
\ncmd{\beqa}{\begin{eqnarray}} \ncmd{\eeqa}{\end{eqnarray}}
\ncmd{\beqan}{\begin{eqnarray*}} \ncmd{\eeqan}{\end{eqnarray*}}
\ncmd{\edpn}[1]{\begin{equation*}\left\lbrace\begin{array}{l}
#1
\end{array}\right.\end{equation*}}
\ncmd{\edp}[3]{\begin{equation}\label{#1}\left\lbrace\begin{array}{#2}
#3
\end{array}\right.\end{equation}}
\ncmd{\barre}[1]{\overline{ #1}}

\ncmd{\espace}{\hspace*{0.5cm}}

\ifpdf
\newcommand{\suff}{png}
\else
\newcommand{\suff}{eps}
\fi

\newcommand{\undessincm}[4]{\begin{figure}[!ht]\begin{center}
    \includegraphics[width=#4cm]{#1.\suff}
     \end{center}
\caption{#2}\label{#3}\end{figure}}

\newcommand{\deuxdessinscm}[5]{\begin{figure}[!ht]\begin{center}
    \includegraphics[width=#5cm]{#1.\suff} \rotatebox[origin=c]{0}{A}\hskip 0.9cm
\includegraphics[width=#5cm]{#2.\suff} \rotatebox[origin=c]{0}{B}
    \end{center}
\caption{#3\label{#4}}\end{figure}}

\newcommand{\quatredessins}[6]{\begin{figure}[!ht]\begin{center}
\includegraphics[width=7cm]{#1.\suff}
\rotatebox[origin=c]{0}{A} \hskip 0.9cm
\includegraphics[width=7cm]{#2.\suff}
\rotatebox[origin=c]{0}{B} \vskip 0.9cm
\includegraphics[width=7cm]{#3.\suff}
\rotatebox[origin=c]{0}{C} \hskip 0.9cm
\includegraphics[width=7cm]{#4.\suff}
\rotatebox[origin=c]{0}{D} \vskip 0.9cm
    \end{center}
\caption{#5\label{#6}}\end{figure}}
\begin{document}
\title{Mathematical and numerical analysis of a model for anti-angiogenic therapy in metastatic cancers.}

\author{Benzekry S\'ebastien\thanks{CMI-LATP, UMR 6632, Universit\'e
de Provence, Technop\^ole Ch\^ateau-Gombert, 39, rue F. Joliot-Curie,
13453 Marseille cedex 13, France.}$\;^{,}$\thanks{Laboratoire de Toxicocinétique et Pharmacocinétique UMR-MD3. 27, boulevard Jean Moulin 13005 Marseille. France. E-mail:\texttt{benzekry@phare.normalesup.org}}}

\date{\today}

\maketitle
\begin{abstract} 
We introduce and analyze a phenomenological model for anti-angiogenic therapy in the treatment of metastatic cancers. It is a structured transport equation with a nonlocal boundary condition describing the evolution of the density of metastasis that we analyze first at the continuous level. We present the numerical analysis of a lagrangian scheme based on the characteristics whose convergence establishes existence of solutions. Then we prove an error estimate and use the model to perform interesting simulations in view of clinical applications. \\

Nous introduisons et analysons un modèle phénoménologique pour les thérapies anti-angiogéniques dans le traitement des cancers métastatiques. C'est une équation de transport structurée munie d'une condition aux limites non-locale qui décrit l'évolution de la densité de métastases. Au niveau continu, des estimations a priori prouvent l'unicité. Nous présentons l'analyse numérique d'un schéma lagrangien basé sur les caractéristiques, dont la convergence nous permet d'établir l'existence de solutions. Nous démontrons ensuite une estimation d'erreur et utilisons le modèle pour produire des simulations intéressantes au regard de possibles applications cliniques.


 \end{abstract}
\tableofcontents
\newpage

\section*{Introduction}

During the evolution of a cancer disease, a fundamental step for the tumor consists in provoking proliferation of the surrounding blood vessels and migration toward the tumour. This process, called tumoral neo-angiogenesis establishes a proper vascular network which ensures to the tumour supply of nutrients and allow the tumor to grow further than 2-3 mm diameter. It is also important in the metastatic process by making possible the spread of cancerous cells to the organism which then can develop in secondary tumors (metastases). Thus, an interesting therapeutic strategy first proposed by J. Folkman \cite{folkman2} in the seventies consists in blocking angiogenesis with the goal to starve the primary tumor by depriving it from nutrient supply. This can be achieved by inhibiting the action of the Vascular Endothelial Growth Factor molecule either with monoclonal antibodies or tyrosine kinase inhibitors. Although the concept of the therapy seems perfectly clear, the practical use of the anti-angiogenic (AA) drugs leaves various open questions regarding to the best temporal administration protocols. Indeed, AA treatments lead to relatively poor efficacy and can even provoke deleterious effects, especially on metastases \cite{paezRibes}. Regarding to these therapeutic failures, it seems that the \textit{scheduling} of the drug plays a major role. Indeed, as shown in the publication \cite{ebos}, different schedules for the same drug can lead to completely different results. Moreover, AA drugs are never given in a monotherapy but always combined with cytotoxic agents (also named chemotherapy) which act directly on the cancerous cells. Again, the scheduling of the drugs seems to be highly relevant \cite{riely} and the optimal combination schedule between these two types of drugs is still a clinical open question. Thus, the complex dynamics of tumoral growth and metastatic evolution have to be taken into account in the design of temporal administration protocols for anti-cancerous drugs.\\
\espace In order to give answers to these questions, various mathematical models are being developed for tumoral growth including the angiogenic process. We can distinguish between two classes of models : mechanistic models (see for instance \cite{billy,floriane}) try to integrate the whole biology of the processes and comprise a large number of parameters; on the other hand phenomenological models aim to describe the tumoral growth without taking into account all the complexity levels (see \cite{swan} for a review and \cite{folkman,dOnofrioLedzewicz,barbo}). Most of these models deal only with growth of the primary tumor but in 2000, Iwata et al. \cite{iwata} proposed a simple model for the evolution of the population of metastases, which was then further studied in \cite{BBHV,devys}. This model did not include the angiogenic process in the tumoral growth and thus could not integrate a description of the effect of an AA drug. We combined it with the tumoral model introduced by Hahnfeldt et al. \cite{folkman} which takes into account for angiogenesis. The resulting partial differential equation is part of the so-called structured population dynamics (see \cite{perthame} for an introduction to the theory) : it is a transport equation with a nonlocal boundary condition. Its mathematical analysis is not classical because the structuring variable is two-dimensional; as far as we know such models have only been studied in the case where one structuring variable is the age and thus has constant velocity (see \cite{tucker,doumic}). This is not the case in our situation and the theoretical analysis of the model without treatment (autonomous case) was performed in \cite{benzekry}.\\
\espace In this paper, we present some mathematical and numerical analysis of the model in the non-autonomous case that is, integrating both cytotoxic and AA treatments and with a general growth field $G$ satisfying the hypothesis that there exists a positive constant $\delta$ such that $G\cdot \nu \geq \delta>0$ where $\nu$ is the normal to the boundary. We first simplify the problem by straightening the characteristics of the equation. We perform some theoretical analysis first at the continuous level (uniqueness and a priori estimates) using the theory of renormalized solutions. Then we introduce an approximation scheme which follows the characteristics of the equation (lagrangian scheme). The introduction of such schemes in the area of size-structured population equations can be found in \cite{angulo} for one-dimensional models. Here, we go further in the lagrangian approach by doing the change of variables straightening the characteristics and discretizing the simple resulting equation, in the case of a  general class of two-dimensional non-autonomous models. We prove existence of the weak solution to the continuous problem through the convergence of this scheme \textit{via} discrete a priori $L^\infty$ bounds and establish an error estimate in the case of more regular data.\\
\espace Finally, we use this scheme to perform various simulations demonstrating the possible utility of the model. First, as a predictive tool for the number of metastases in order to refine the existing classifications of cancers regarding to metastatic aggressiveness. Secondly, the model can be used to test various temporal administration protocols of AA drugs in monotherapy or combined with a cytotoxic agent.

\section{Model}

The model is based on the approach of \cite{iwata,BBHV,devys} to describe the evolution of a population of metastases represented by its density $\rho(t,X)$ with $X$ being the structuring variable, here two-dimensional $X=(x,\tt)$ with $x$ the size (=number of cells) and $\tt$ the so-called angiogenic capacity. It is a partial differential equation of transport type. The behavior of each individual of the population (metastasis), that is the growth rate $G(t,X)$ of each tumor is taken from \cite{folkman} and is designed to take into account for the angiogenic process, as well as the effect of both anti-angiogenic (AA) and cytotoxic drugs (CT). Its expression will be established in the following subsection. The model writes

\edp{equationNonAutonome}{ll}{\d_t\rho(t,X) + \div(\rho(t,X) G(t,X))=0  & \forall (t,X)\in ]0,T[\times \Om\\
     -\Gn(t,\sig) \rho(t,\sig)=N(\sig)\int_\Om \beta(X)\rho(t,X)dX + f(t,\sig) & \forall (t,\sig) \in]0,T[\times\d\Om\\
     \rho(0,X)=\rho^0(X) & \forall X \in \Om .
}
where $\Om$, the birth rate $\beta(X)$, the repartition along the boundary $N(\sig)$ and the source term $f(t,\sig)$ will be specified in the sequel, $T$ is a positive time and $\nu$ is the unit external normal vector to the boundary $\d\Om$.
\subsection{The model of tumoral growth under angiogenic control (Hahnfeldt et al. \cite{folkman})}
Let $x(t)$ denote the size (number of cells) of a given tumor at time $t$. The growth of the tumor is modeled by a gompertzian growth
rate modified by a death term describing the action of a CT. The equation is :
\begin{equation}\label{g1}
\frac{dx}{dt}=g_1(t,x)=ax\ln\left(\frac{\theta}{x}\right) - h \gamma_{C}(t)H(x-x_{min}),
\end{equation}
where $a$ is a parameter representing the velocity of the growth, $\tt$ the carrying capacity of the environment, and the term $h\gamma_{C}(t)H(x-x_{min})$ stands for the effect of a cytotoxic drug, where $\gamma_C$ is the concentration of the CT, $x_{min}$ is a minimal size for the drug to be effective ($x_{min}\geq 1$) and the function $H$ is a regularization of the Heaviside function (for example $H(t)=1/2+1/2\tanh(t/K)$, with $K$ being a parameter controlling the slope at $0$), in order to avoid regularity issues in the analysis. The idea is now to take $\tt$ as a variable of the time, representing the degree of vascularization of the tumor and called "angiogenic capacity". The variation rate for $\tt$ derived in \cite{folkman} is :
\begin{equation}\label{g2}
\frac{d\theta}{dt} = g_2(t,x,\tt)=cx-d\tt {x}^{\frac{2}{3}} - e \gamma_{A}(t)H(\tt - \tt_{min}),
\end{equation}
where the terms $cx$ and $-d\tt x^{2/3}$ represent respectively the endogenous stimulation and inhibition of the vasculature and $e \gamma_{A}(t)H(\tt-\tt_{min})$ is the effect of an anti-angiogenic drug. The factor $2/3$ comes from the analysis of \cite{folkman} which concluded that the ratio of the stimulation rate over the inhibition one should be homogeneous to the tumoral radius to the square. 
In the figure \ref{PlansPhase}, we present some numerical simulations of the phase plan of the system \eqref{g1}, \eqref{g2}.

\begin{figure}[!h]\begin{center}
    \includegraphics[width=6cm]{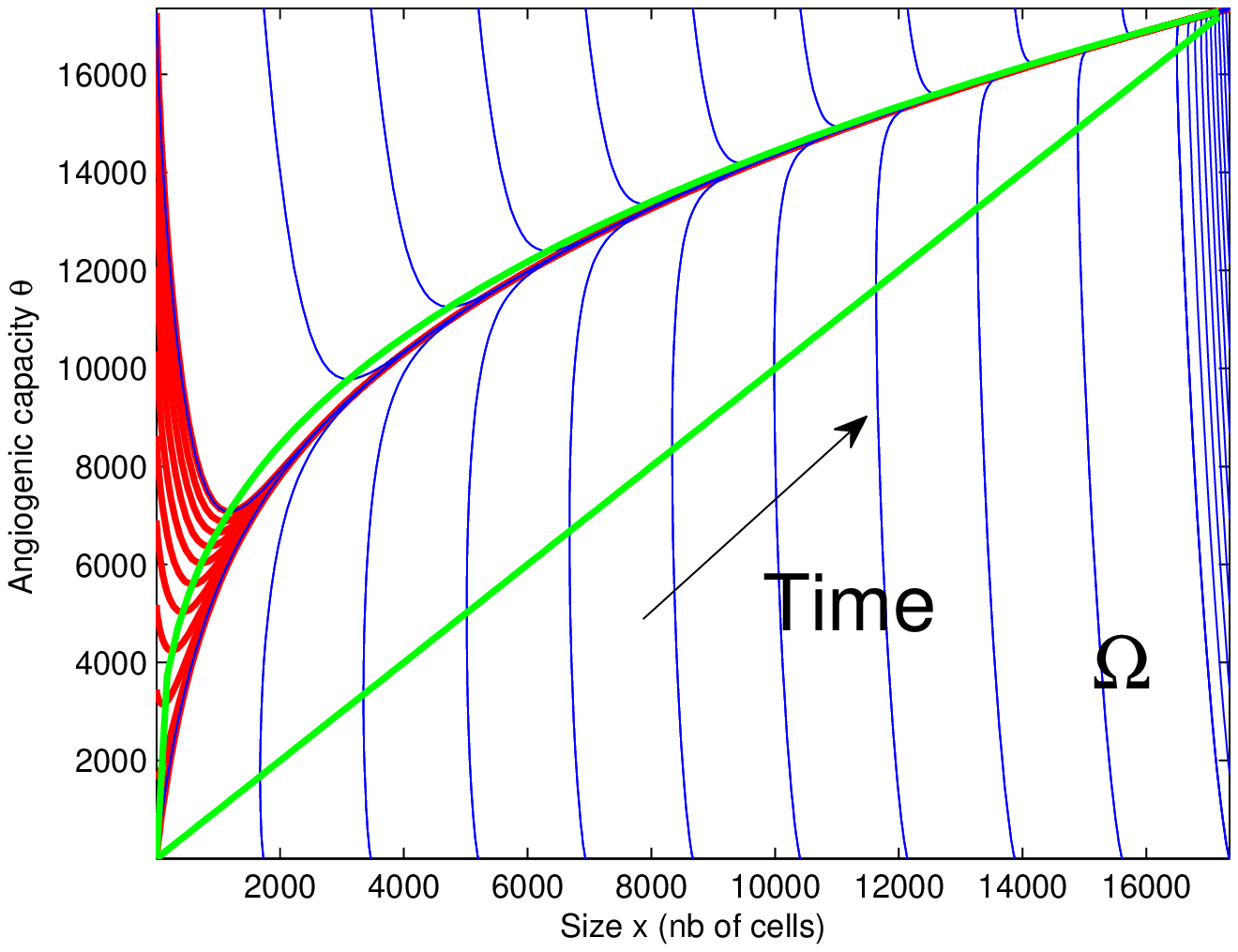} \rotatebox[origin=c]{0}{A}\hskip 0.9cm
\includegraphics[width=6cm]{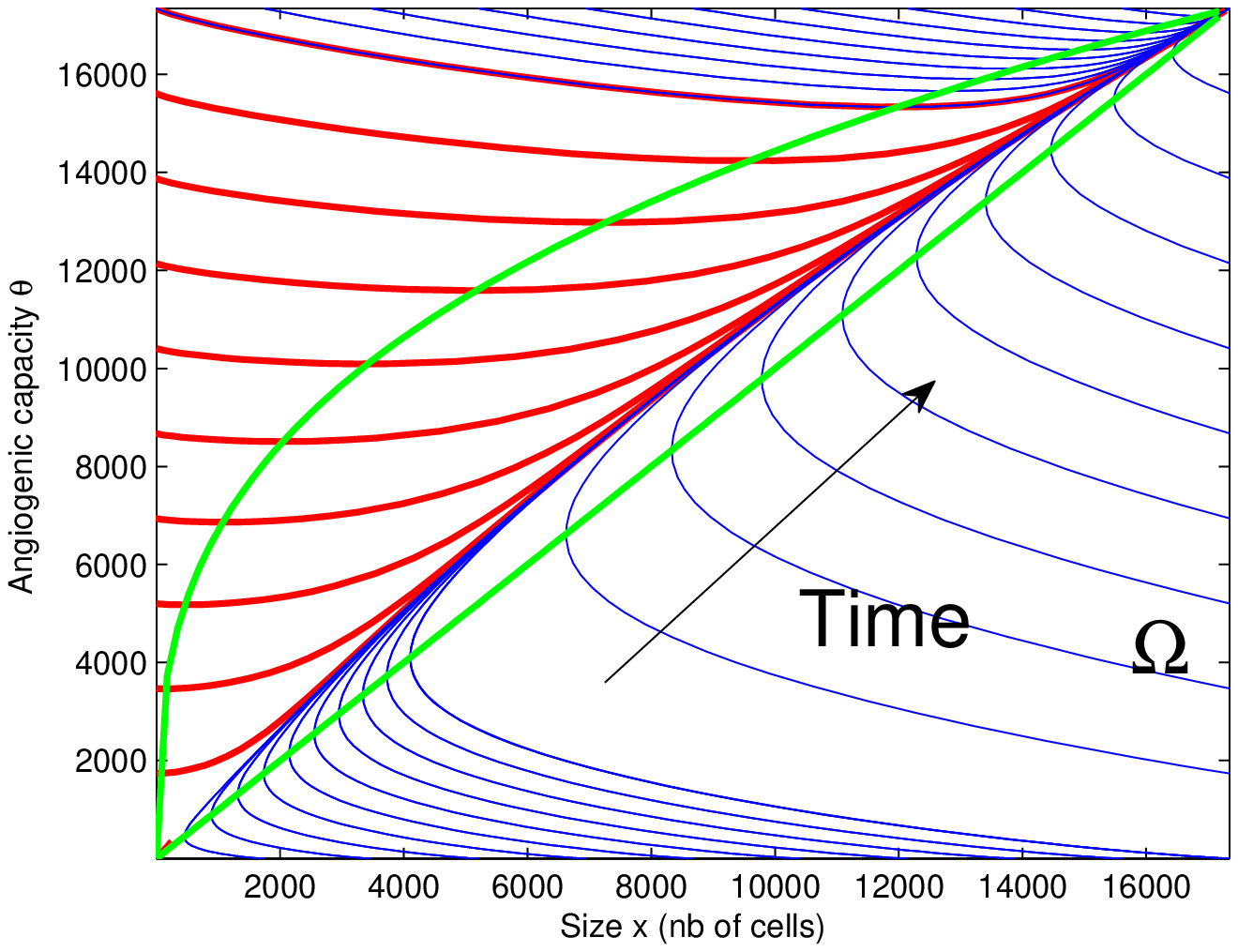} \rotatebox[origin=c]{0}{B}
    \end{center}
\caption{Two phase plans of the system \eqref{g1}-\eqref{g2}, for different values of the parameters, without treatment ($h=e=0$). In green, the nullclines. In both, $b:=\left(\frac{c}{d}\right)^{\frac{3}{2}}=17347$. A. Parameters from \cite{folkman} : $a=0.192,\;c=5.85,\;d=8.73\times 10^{-3}$. B. $a=0.192,\;c=0.1,\;d=1.4923\times 10^{-4}$\label{PlansPhase}}\end{figure}

Following \cite{folkman}, we assume a one compartmental pharmacokinetic for the AA and do the same for the CT (in \cite{folkman} there is no CT). We also assume that the drugs are administered as boli. This gives
$$\gamma_{A}(t)=\sum_{i=1}^N D_{A}e^{-clr_{A}(t-t_i^A)}H(t-t_i^A)$$
where the $t_i^A$ are the administration times of the AA, $D_{A}$ is the administered dose and $clr_{A}$ the clearance. The expression for the CT is the same, with C instead of A.

\subsection{Renewal equation for the density of metastasis}\label{equationRenouv}

We denote $X=(x,\tt)$ and $G(t,X)=\left(g_1(t,x,\tt), g_2(t,x,\tt)\right)$. We define $b=\left(\frac{c}{d}\right)^{\frac{3}{2}}$ and $\Om=(1,b)\times (1,b)$
where $b$ is the maximal reachable size and angiogenic capacity for $(x,\tt)$ solving the system \eqref{g1},\eqref{g2} with initial size $1$ (see \cite{dOnofrio_Gandolfi} for a study of this system without the CT term). We consider that each tumor is a particle evolving in $\Om$ with the velocity $G$. Writing a balance law for the density $\rho(t,X)$ we have
$$\d_t \rho + \div(\rho G)=0,\quad \forall (t,X)\in ]0,T[\times\Om$$
that we endow with an initial condition $\rho^0 \in L^\infty(\Om)$. \\
\espace Metastasis do not only grow in size and angiogenic capacity, they are also able to emit new metastasis. We denote by
$\mathbf{b}(\sig,x,\tt)$ the birth rate of new metastasis with size and angiogenic capacity $\sig\in \d\Om$ by metastasis of size $x$
and angiogenic capacity $\tt$, and by $f(t,\sig)$ the term corresponding to metastasis produced by the primary tumor. Expressing
the equality between the number of metastasis arriving in $\Om$ per unit time (l.h.s in the following equality) and the total rate of new metastasis created by both the primary tumor and metastasis themselves (r.h.s.), we should have for all $t>0$
\begin{equation}\label{condLimInt}
-\int_{\d\Om}\rho(t,\sig)G(t,\sig)\cdot\nu d\sig=\int_{\d\Om}\int_{\Om}\mathbf{b}(\sig,X)\rho(t,X)dX + f(t,\sig) d\sig.
\end{equation}
We assume that the emission rate of the primary and secondary tumors are equal and thus take $f(t,\sig)=\mathbf{b}(\sig,X_p(t))$ where $X_p(t)$ represents the primary tumor and solves the ODE system \eqref{g1}-\eqref{g2}. We also assume that the new metastasis created have size $x=1$ and that there is no metastasis of maximal size $b$ nor maximal or minimal angiogenic capacity because they should come from metastasis outside of $\Om$ since $G$ points inward all along $\d\Om$. An important feature of the model is to assume that the vasculature of the neo-metastasis is independent from the one which emitted it. This means that $\mathbf{b}(\sig,X)=N(\sig)\beta(x,\tt)$ with $N(\sig)$ having its support in $\{\sig\in \d\Om; \;\sig=(1,\tt),\,1\leq\tt\leq b\}$ and describing the angiogenic distribution of the metastasis at birth. We assume it to be uniformly centered around a mean value $\tt_0$, thus we take $N(1,\tt)=\frac{1}{2\Delta \tt}\mathbf{1}_{\tt\in[\tt_0 - \Delta \tt,\tt_0 + \Delta \tt]}$, with $\Delta \tt$ a parameter of dispersion of the new metastasis around $\tt_0$. Following the modeling of \cite{iwata} for the colonization rate $\beta$ we take 
$$\beta(x,\tt)=m x^\alpha,$$
with $m$ the colonization coefficient and $\alpha$ the so-called fractal dimension of blood vessels infiltrating the tumor. The parameter $\alpha$ expresses the geometrical distribution of the vessels in the tumor. For example, if the vasculature is superficial then $\alpha $ is assigned to $2/3$ thus making $x^\alpha$ proportional to the area of the surface of the tumor (assumed to be spheroidal). Else if the tumor is homogeneously vascularised, then $\alpha$ is supposed to be equal to 1. Assuming the equality of the integrands in \eqref{condLimInt} in order to have the equality of the integrals, we obtain the boundary condition of \eqref{equationNonAutonome}.

\section{Analysis at the continuous level}
In the autonomous case, that is when $G$ depends only on $X$ and there is no treatment, the analysis of the equation \eqref{equationNonAutonome} has been performed in \cite{benzekry}. It was proven the existence, uniqueness, regularity and asymptotic behavior of solutions. We present now some analysis on the equation \eqref{equationNonAutonome} with a more general growth field $G$ than the one defined in the section \ref{equationRenouv}.\\
\espace Let $\Om$ be a bounded domain in $\R^2$, with $\d\Om$ being piecewise $\C^1$ except in a finite number of points. Let $G:\R\times \barre{\Om}\rightarrow \R^2$ be a $\C^1$ vector field. We make the following assumption on $G$ : 
\begin{equation}\label{hypG}
\exists \, \delta>0,\; \Gn(t,\sig)\geq \delta>0 \quad \forall \;0\leq t \leq T,\;\sig \in\d\Om.
\end{equation}
We do the following assumptions on the data :
\begin{equation}\label{hypDonnees}
 \rho^0 \in L^\infty(\Om),\;\beta\in L^\infty(\Om),\;N\in L^\infty(\d\Om),\,N\geq 0,\,\int_{\d\Om}N(\sig)d\sig=1,\; f\in L^\infty(]0,T[\times \d\Om).
\end{equation}

\begin{rem} In the case of $G$ being the one of the section \ref{equationRenouv} if there is no treatment (that is, if $e=h=0$, or $t\leq t_1$) then we don't have $\Gn(t,\sig)\geq m >0$ all along the boundary since $G$ vanishes at the point $(b,b)$. But since the problem was solved in this case (see \cite{benzekry}) we consider that the time $0$ is the starting time of the treatment and that $e$ or $h$ is positive, which makes the assumption \eqref{hypG} true.\\
\end{rem}

\begin{defn}[Weak solution]
  We say that $\rho \in L^\infty(]0,T[\times \Om)$ is a weak solution of the problem \eqref{equationNonAutonome} if for all test function $\phi$ in $\C^1([0,T]\times\bOm)$ with $\phi(T,\cdot)=0$
\begin{align}\label{formulationFaible}
 \int_0^T \int_\Om  \rho(t,X)  \left[\d_t \phi(t,X)   + G(t,X)\cdot\nabla\phi(t,X)\right]dXdt  & + \int_\Om \rho^0(X)\phi(0,X)dX \notag\\
  & + \int_0^T\int_{\d\Om}\left\lbrace N(\sig)B(t,\rho) + f(t,\sig)\right\rbrace\phi(t,\sig)=0
\end{align}
where we denoted $B(t,\rho):=\int_\Om \beta(X)\rho(t,X)dX$.
\end{defn}

\begin{rem}\label{remTestLip}
 By approximating a Lipschitz function by $\C^1$ ones, it is possible to prove that the definition of weak solutions would be equivalent with test functions in $W^{1,\infty}([0,T]\times\bOm)$ vanishing at time $T$.
\end{rem}

\subsection{Change of variables}
Let $X(t;\tau,\sig)$ be the solution of the differential equation
\begin{equation}\label{EDO}
\left\lbrace\begin{array}{c}\frac{d}{d t}X=G(t,X)\\
X(\tau;\tau,\sig)=\sig \end{array}\right. .
\end{equation}
For each time $t>0$, we define the entrance time $\tau^t(X)$ and entrance point $\sig^t(X)$ for a point $X\in \Om$ :
$$\tau^t(X):=\inf\{0\leq\tau\leq t;\;X(\tau;t,X)\in \Om\},\;\sig^t(X):=X(\tau^t(X);t,X).$$
We consider the sets 
$$\Om_1^t=\{X\in \Om;\;\tau^t(X)> 0\},\;\Om_2^t=\{X\in \Om;\;\tau^t(X)= 0\}$$
and
$$Q_1:=\{(t,X)\in[0,T]\times\bOm;\;X\in\bOm_1^t\},\;Q_2:=\{(t,X)\in[0,T]\times \bOm;\;X\in \bOm_2^t\}.$$
We also define $\tilda{Q_1}:=\{(t,\tau,\sig);\;0\leq \tau\leq t \leq T,\;\sig\in\d\Om\}=X^{-1}(Q_1)$ and notice that $$\Sigma_1:=[0,T]\times\d\Om=\{(t,X);\;\tau^t(X)=0\},\text{ and } \Sigma_2=\{(t,X(t;0,\sig));\;0\leq t \leq T,\;\sig\in\d\Om\}=\{(t,X);\;\tau^t(X)=0\}.$$
See figure \ref{redressCaract} for an illustration. We can now introduce the changes of variables that we will constantly use in the sequel.
\begin{prop}[Change of variables]\label{propChgtVar}
 The maps
$$X_1:\begin{array}{ccc} \tQ_1 & \rightarrow & Q_1 \\
		       (t,\tau,\sig) & \mapsto      & X(t;\tau,\sig)
  \end{array}
\text{ and }
X_2 : \begin{array}{ccc} [0,T]\times\bOm & \rightarrow & Q_2 \\
		       (t,Y) & \mapsto      & X(t;0,Y)
  \end{array}
$$
are bilipschitz. The inverse of $X_1$ is $(t,X)\mapsto(t,\tau^t(X),\sig^t(X))$ and the inverse of $X_2$ is $(t,X)\mapsto (t,Y(X))$ with $Y(X)=X(0;t,X)$. Denoting $J_1(t;\tau,\sig)=\left|\det(DX_1)\right|$ and $J_2(t;Y)=\left|\det(DX_2)\right|$, we have :
\begin{equation}\label{formule_jacobien}
J_1(t;\tau,\sig)=|G(\tau,\sig)\cdot\nuu (\sig)|e^{\int_\tau^t \div \,
G(u,X(u;\tau,\sig))du}\text{ and }J_2(t;Y)=e^{\int_0^t \div \,
G(u,X(u;0,Y))du}
\end{equation}
\end{prop}

We refer to the appendix for the proof of this result and to the figure \ref{redressCaract} for an illustration.

\begin{figure}
 \input{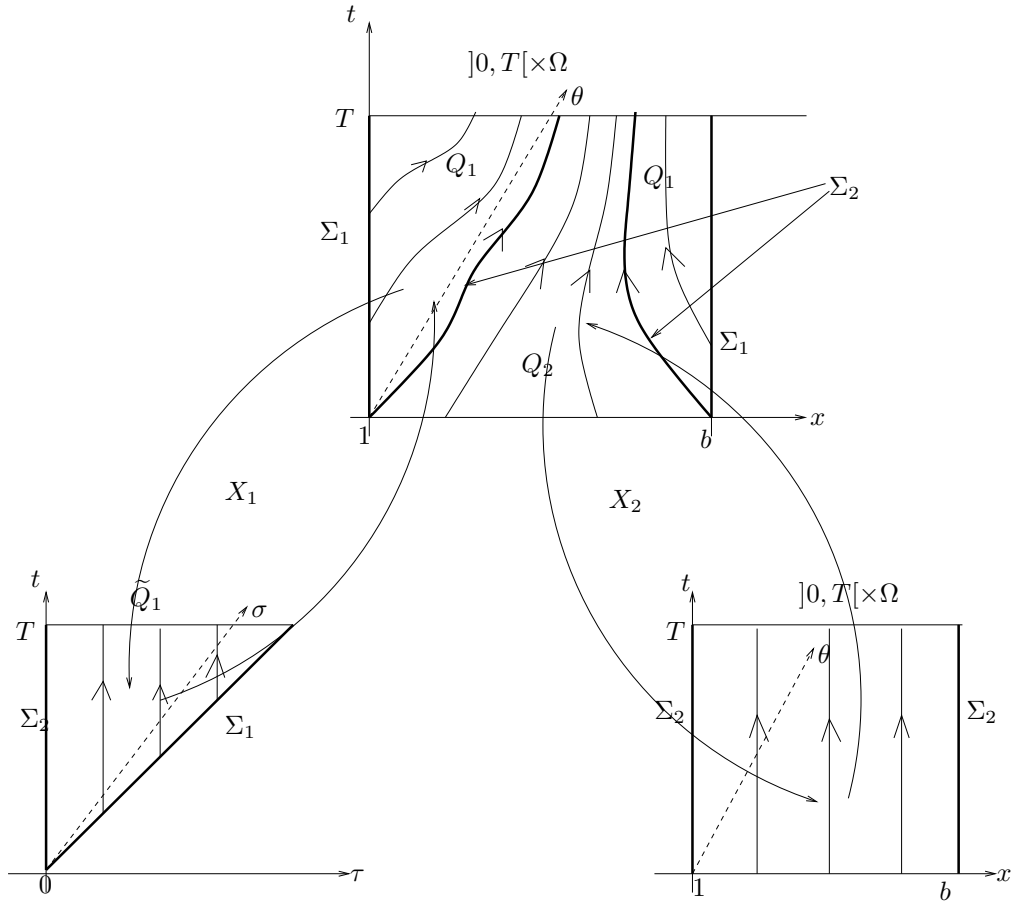}
\caption{The two changes of variables $X_1$ and $X_2$ (represented only on the plane $\tt=1$).}\label{redressCaract}
\end{figure}

Using these changes of variables we can write for a function $f\in L^1(]0,T[\times\Om)$
$$\int_0^T\int_\Om f(X)dX=\int_0^T\int_0^t\int_{\d\Om}f(X_1(t;\tau,\sig))J_1(t;\tau,\sig)d\sig d\tau + \int_0^T\int_\Om f(X_2(t;0,Y))J_2(t;Y)dY.$$
We want to decompose the equation \eqref{equationNonAutonome} into two subequations : one for the contribution of the boundary term and one for the contribution of the initial condition since they are ``independent''. Defining
\begin{equation}\label{expressionTrho12}
\trho_1(t,\tau,\sig):=\rho(t,X(t;\tau,\sig))J_1(t;\tau,\sig)\text{ and } \trho_2(t,y):=\rho(t,X(t;0,Y))J_2(t;Y) 
\end{equation}
we have, when the solution is regular : $\d_t\trho_1=\left(\d_t\rho+\div\left(\rho G\right)\right)J_1=0$ and the same for $\trho_2$. It is thus natural to introduce the following equations 
\edp{equationTrho1}{ll}{
\d_t\trho_1(t,\tau,\sig) =0 & 0<\tau\leq t <T,\;\sig\in\d\Om  \\
\trho_1(\tau,\tau,\sig)=N(\sig)\tB(t,\trho_1,\trho_2)+f(t,\sig) & 0<\tau<T,\;\sig\in\d\Om
}
where we denoted 
$$\tB(t,\trho_1,\trho_2)=\int_0^t\int_{\d\Om}\beta(X(t;\tau,\sig))\trho_1(t,\tau,\sig)d\sig d\tau + \int_\Om \beta(X(t;0,Y))\trho_2(t,Y)dY,$$
and
\edp{equationTrho2}{ll}{
\d_t\trho_2 =0 & t>0,\;Y\in\Om \\
\trho_2(0;Y)=\rho^0(Y) & Y\in \Om
.}

We precise the definition of weak solutions to these equations.
\begin{defn}\label{defSolRho1Rho2}
 We say that a couple $(\trho_1,\trho_2) \in L^\infty(\tQ_1)\times L^\infty(]0,T[\times \Om)$ is a weak solution of the equations \eqref{equationTrho1}-\eqref{equationTrho2} if for all $\tphi_1 \in \C^1(\tQ_1)$ with $\tphi_1(T,\cdot)=0$ we have :
\begin{align}\label{formulationFaibleTrho1}
 \int_0^T\int_0^t\int_{\d\Om}\trho_1(t,\tau,\sig)\d_t\tphi_1(t,\tau,\sig)d\sig d\tau dt +\int_0^T\int_{\d\Om}\left\lbrace N(\sig)\tB(t,\trho_1,\trho_2)+f(t,\sig)\right\rbrace\tphi_1(t,t,\sig)  =0
,\end{align}
and for all $\tphi_2 \in \C^1([0,T]\times\bOm)$ with $\tphi_2(T,\cdot)=0$ we have
\begin{equation}\label{formulationFaibleTrho2}
 \int_0^T\int_\Om \trho_2(t,Y)\d_t\tphi_2(t,Y)dt + \int_\Om \rho^0(Y)\tphi_2(0,Y)dY =0.
\end{equation}
\end{defn}
\begin{rem}
 If $\trho_1$ is a regular function which solves \eqref{equationTrho1}, then the weak formulation is satisfied since we have :
\begin{align*}
\int_0^T\int_0^t \int_{\d\Om} & \trho_1(t,\tau,\sig)\d_t\tphi_1(t,\tau,\sig)d\sig d\tau dt =   \underset{=0}{\underbrace{\int_0^T\int_{\d\Om}\tphi_1(T,\tau,\sig)\trho_1(T,\tau,\sig)d\tau d\sig}} \\
 & -  \int_0^T\int_0^t \tphi_1(t,\tau,\sig)\d_t \trho_1(t,\tau,\sig)d\sig d\tau dt - \int_0^T\int_{\d\Om}\tphi_1(t,t,\sig)\trho_1(t,t,\sig)d\sig dt.
\end{align*}
\end{rem}

We prove now the following theorem, establishing the equivalence between the problem \eqref{equationNonAutonome} and the problem \eqref{equationTrho1}-\eqref{equationTrho2}.

\begin{thr}[Equivalence between problem \eqref{equationNonAutonome} and problem \eqref{equationTrho1}-\eqref{equationTrho2}]\label{propRhoRho1Rho2}
Let $\rho \in L^\infty(]0,T[\times \Om)$ be a weak solution of the equation \eqref{equationNonAutonome}. Then $(\trho_1,\trho_2)$ given by \eqref{expressionTrho12} is a weak solution of \eqref{equationTrho1}-\eqref{equationTrho2}.
Conversely, if $\trho_1$ and $\trho_2$ are weak solutions of \eqref{equationTrho1} and \eqref{equationTrho2}, then the function defined by \begin{equation}\label{rhoRho1Rho2}
\rho(t,X):=\trho_1(t,\tau^t(X),\sig^t(X))J_1^{-1}(t,\tau^t(X),\sig^t(X))\mathbf{1}_{X\in\Om_1^t} + \trho_2(t,Y(X))J_2^{-1}(t,Y(X))\mathbf{1}_{X\in\Om_2^t}                                                                                                                                            
\end{equation}
is a weak solution of \eqref{equationNonAutonome}.
\end{thr}

\begin{proof}~\\
\espace $\bullet$\textit{ Direct implication.} Let $\rho$ be a weak solution of the equation \eqref{equationNonAutonome}. We will prove that $\trho_2$ defined by \eqref{expressionTrho12} solves \eqref{equationTrho2}. Let $\tphi_2\in\C^1([0,T]\times \bOm)$ with $\tphi_2(T,\cdot)=0$. We define for $X\in Q_2$ $\phi_2(t,X):=\tphi_2(t,Y(X)) \in \Winf(Q_2)$ and we intend to extend it in a Lipschitz function of $[0,T]\times\bOm$ so that we can use it as a test function in the weak formulation for $\rho$ (see remark \ref{remTestLip}). We define, for $(t,\tau,\sig)\in \tQ_1$, $\tphi_1^\eps(t,\tau,\sig)=\tphi_2(t,\sig)\zeta_\eps(\tau)$
with $\zeta_\eps(\tau)$ being a truncature function in $\C^1([0,+\infty[)$ such that $0\leq \zeta_\eps\leq 1,\;\zeta_\eps(0)=1,\;\zeta_\eps(\tau)=0$ for $\tau\geq \eps$. Then $\tphi_1^\eps \in W^{1,\infty}(\tQ_1)$ and we set $\phi_1^\eps(t,X):=\tphi_1^\eps(t,\tau^t(X),\sig^t(X))\in\Winf(Q_1)$ since $\tau^t(X)$ and $\sig^t(X)$ are Lipschitz from proposition \ref{propChgtVar}. We define then
$$\phi^\eps:=\left\lbrace\begin{array}{l}
   \phi_1^\eps \text{ on }Q_1 \\
  \phi_2 \text{ on }Q_2
  \end{array}\right. .
$$
The function $\phi^\eps$ is Lipschitz on $Q_1$, Lipschitz on $Q_2$ and $\phi^\eps\in\C([0,T]\times\bOm)$ since $Q_1\cap Q_2=\{(t,X);\,\tau^t(X)=0\}$. Thus $\phi^\eps\in \Winf([0,T]\times\bOm)$ with $\phi^\eps(T,\cdot)=0$. Using $\phi^\eps$ as a test function in \eqref{formulationFaible}, we have
\begin{align*}
\int_{Q_1} & \rho[\d_t\phi_1^\eps + G\cdot\nabla \phi_1^\eps]dXdt + \int_0^T\int_{\d\Om}\left\lbrace N(\sig) B(t,\rho)+f(t,\sig)\right\rbrace\phi_1^\eps(t,\sig)dtd\sig \\
 & + \int_{Q_2}\rho[\d_t\phi_2 + G\cdot \nabla\phi_2]dXdt + \int_\Om \rho^0(X)\phi_2(0,X)dX=0=I^1_\eps + I_2.
\end{align*}
By doing the change of variables $X_1$ in the term $I_1^\eps$ and noticing that $\phi_1^\eps(t,\sig)=\tphi_1^\eps(t,t,\sig)=\tphi_2(t,\sig)\zeta_\eps(t)$, we obtain 
$$I_1^\eps=\int_0^T\int_0^t\trho_1(t,\tau,\sig) \d_t\tphi_1(t,\sig)\zeta^\eps(\tau)d\sig d\tau dt + \int_0^T\int_{\d\Om}B(t,\rho)\tphi_2(t,\sig)\zeta^\eps(t)d\sig \xrightarrow[\eps\rightarrow 0]{}0.$$
Now doing the change of variables $X_2$ in the second term $I_2$ and noticing that $\d_t \tphi_2(t,Y)=\d_t(\phi_2(t,X(t;0,Y)))=\d_t\phi_2(t,X(t;0,Y))+G(t,X(t;0,Y))\cdot\nabla\phi_2(t,X(t;0,Y))$ gives the result. The equation on $\trho_1$ is proved in the same way.\\
\espace $\bullet$ \textit{Reverse implication.} Let $\trho_1$ and $\trho_2$ be solutions of \eqref{equationTrho1} and \eqref{equationTrho2} respectively. Define $\rho(t,X)$ by \eqref{rhoRho1Rho2}, and consider a test function $\phi \in \C^1([0,T]\times \bOm)$ with $\phi(T,\cdot)=0$. Then $\phi_1:=\phi_{|Q_1} \in \C^1(Q_1)$, with $\phi_1(T,\cdot)=0$, thus $\tphi_1(t,\tau,\sig):=\phi_1(t,X_1(\tau,\sig))$ is valid as a test function in the weak formulation of \eqref{equationTrho1}. In the same way $\tphi_2(t,Y):=\phi_2(t,X_2(Y))$ with $\phi_2:=\phi_{|Q_2}$ is valid as a test function for \eqref{equationTrho2}. Thus we have
\begin{align*}
\int_{\tilda{Q}_1} & \trho_1(t,\tau,\sig)\d_t\tphi_1(t,\tau,\sig) d\sig d\tau dt + \int_0^T\int_{\d\Om}\tB(t,\trho_1,\trho_2) \tphi_1(t,t,\sig)d\sig dt \\
& + \int_0^T\int_\Om \trho_2(t,y) \d_t\tphi_2(t,y)dtdy + \int_\Om \rho^0(y)\tphi_2(0,y)dy = 0
\end{align*}
Doing the changes of variables gives the weak formulation of \eqref{equationNonAutonome}.
\end{proof}

This theorem simplifies the structure of the problem \eqref{equationNonAutonome}. In some sense, it formalizes the method of characteristics in the framework of weak solutions for our problem. The characteristics are straightened (see figure \ref{redressCaract}) and the directional derivative along the field $(t,G)$ is transformed in only a time derivative. Moreover, integrating the jacobians (which contains the transformation of areas) in the definitions of $\trho_1$ and $\trho_2$, these functions are constant in time. The continuous analysis and discrete approximation of the problem \eqref{equationNonAutonome} is thus simplified.

\subsection{\textit{A priori} continuous estimates and uniqueness}

In order to obtain \textit{a priori} properties on the solutions of the equation, we will use the theory of renormalized solutions first initiated by DiPerna-Lions \cite{dipernalions} in the case of $\R^n$ and further developed by Boyer \cite{franck} in the case of a bounded domain. Let us first recall the result that we will use, which can be found in \cite{franck}. We need to introduce the following measure on $]0,T[\times \d\Om$ : $d\mu_{G}=(G\cdot \nu)dtd\sig$

\begin{prop}[Renormalization property]\label{propRenormalisation}
 Let $\rho\in L^\infty(]0,T[\times\Om)$ be a solution, in the distribution sense, to the equation :
\begin{equation}\label{equationTransport}
\d_t\rho + \div(\rho G)=0.
\end{equation}
\begin{enumerate}
 \item[(i)] The function $\rho$ lies in $\C([0,T];L^p(\Om))$, for any $1\leq p <\infty$. Furthermore, $\rho$ is continuous in time with values in $L^\infty(\Om)$ weak-$\ast$.
\item[(ii)] There exists a function $\gamma\rho \in L^\infty(]0,T[\times \d\Om;|d\mu_G|)$ such that for any $h \in \C^1(\R)$, for any test function $\phi\in \C^1([0,T]\times \barre{\Om})$, and for any $[t_0,t_1]\subset[0,T]$, we have
\begin{align}\label{renormalisation_trace}
 \int_{t_0}^{t_1}\int_\Om h(\rho)(\d_t \phi & + \div(G\phi))dtdX + \int_\Om h(\rho(t_0))\phi(t_0)dX - \int_\Om h(\rho(t_1))\phi(t_1)dX  \notag\\
  & -\int_{t_0}^{t_1}\int_{\d\Om}h(\gamma\rho)\phi \Gn dtd\sig - \int_{t_0}^{t_1}\int_\Om h'(\rho)\rho\div(G)\phi dtdX =0
\end{align}
\end{enumerate}
\end{prop}

\begin{rem}\label{remValAbs}~\\
 \espace $\bullet$ By approximating the function $s\mapsto |s|$ by $\C^1$ functions, it is possible to show that the formula \eqref{renormalisation_trace} stands with $h(s)=|s|$. \\
\espace $\bullet$ The second point of the proposition implies in particular that $h(\rho)$ has a trace which is $h(\gamma\rho)$. \\
\espace $\bullet$ In \cite{franck}, this proposition is proved in the case of a much less regular field $G$ but with the technical assumption that $\div G=0$, which is not the case here. Though, the proof can be extended to our case.
\end{rem}
Thanks to this result, we can prove the following proposition.
\begin{prop}[Continuous \textit{a priori} estimates]\label{propEstimCont}
 Let $\rho \in L^\infty(]0,T[\times \Om)$ be a weak solution of the equation \eqref{equationNonAutonome}. The following estimates stand 
\begin{equation}\label{estimationL1}
 ||\rho(t,\cdot)||_{L^1(\Om)}\leq e^{t||\beta||_{L^\infty}}||\rho^0||_{L^1(\Om)} + \int_0^t e^{(t-s)||\beta||_{L^\infty}}\int_{\d\Om}|f(s,\sig)|d\sig ds
\end{equation}
and
\begin{align}\label{estimationLinf}
||\rho||_{L^\infty(]0,T[\times\Om)} \leq C_\infty
\end{align}
with 
$$C_\infty=\left(||N||_{L^\infty}||\beta||_{L^\infty}||\rho||_{L^\infty} + ||f||_{L^\infty}\right)||G||_{L^\infty}e^{T||\div \,
G||_{L^\infty}} + ||\rho^0||_{L^\infty}e^{T||\div \,
G||_{L^\infty}} $$
\end{prop}
\begin{proof}~\\
\espace $\bullet$\textit{ Estimate in $L^1$.} Let $\rho$ be a weak solution of the equation \eqref{equationNonAutonome}. Then in particular it solves \eqref{equationTransport} in the sense of distributions. Thus the proposition \ref{propRenormalisation} applies and gives a trace $\gamma\rho \in L^\infty(]0,T[\times\d\Om;|d_{\mu G}|)$. Now, by using \eqref{renormalisation_trace} with $h(s)=s$ and the definition of weak solutions to the equation \eqref{equationNonAutonome} we have that for all $\phi \in \C^1_c([0,T[\times \barre{\Om})$ 
$$\int_0^T\int_{\d\Om}\gamma\rho(t,\sig)\phi(t,\sig)G(t,\sig)\cdot \nu d\sig dt = \int_0^T\int_{\d\Om}\left\lbrace N(\sig)\int_\Om \beta(X) \rho(t,X)dX +f(t,\sig)\right\rbrace\phi(t,\sig)d\sig dt$$
which gives 
\begin{equation}\label{traceRenouv}
-\gamma\rho(t,\sig)G(t,\sig)\cdot \nu=N(\sig)\int_\Om \beta(X) \rho(t,X)dX +f(t,\sig),\quad a.e.
\end{equation}
In view of the remark \ref{remValAbs}, we know that $|\rho|$ is also a weak solution to the equation \eqref{equationNonAutonome}, with initial data $|\rho^0|$ and boundary data $\left| N(\sig) B(t,\rho) + f(t,\sig)\right|$. By integrating this equation on $\Om$ and using the divergence formula, we obtain in the distribution sense :
$$\frac{d}{dt}\int_\Om |\rho(t,X)|dX =- \int_{\d\Om}G(t,\sig)\cdot\nu|\gamma\rho(t,\sig)|d\sig=\int_{\d\Om}\left|N(\sig) B(t,\rho) +f(t,\sig)\right| d\sig$$
and thus
$$\frac{d}{dt}\int_\Om |\rho(t,X)|dX \leq ||\beta||_{\infty}\int_\Om |\rho(t,X)|dX + |f(t,\sig)|.$$
A Gronwall lemma concludes. \\
\espace $\bullet$ \textit{Estimate in $L^\infty$.} Using the proposition \ref{propRhoRho1Rho2}, we have $\trho_1$ and $\trho_2$ solving \eqref{equationTrho1} and \eqref{equationTrho2}. By doing the changes of variables, using the definitions of $\trho_1$ and $\trho_2$ and the formulas \eqref{formule_jacobien}, we see that 
$$||\rho(t,\cdot)||_{L^1(\Om)}=||\trho_1(t,\cdot)||_{L^1(]0,t[\times \d\Om)}+||\trho_2(t,\cdot)||_{L^1(\Om)},\quad \forall \,t>0$$
$$||\rho||_{L^\infty(]0,T[\times\Om)}\leq ||\trho_1||_{L^\infty(\tQ_1)}||G||_{L^\infty(\d\Om)}e^{T||\div \,
G||_{\infty}} + ||\trho_2||_{L^\infty(]0,T[\times\Om)}e^{T||\div \,
G||_{\infty}}$$
But solving explicitely the equation \eqref{equationTrho1}, we have 
\begin{align*}
|\trho_1(t,\tau,\sig)| & =|\trho_1(\tau,\tau,\sig)|=\left|N(\sig)\tB(t,\trho_1,\trho_2)+f(t,\sig)\right| \\
	 & \leq ||N||_\infty ||\beta||_\infty (||\trho_1(\tau,\cdot)||_{L^1} + ||\trho_2(\tau,\cdot)||_{L^1}) + ||f||_{L^\infty}\\
	 & \leq ||N||_\infty ||\beta||_\infty ||\rho(\tau,\cdot)||_{L^1} + ||f||_{L^\infty}
\end{align*}
On the other hand, for $\trho_2$ we have $||\trho_2||_{L^\infty(]0,T[\times\Om)}=||\trho_2(0)||_{L^\infty(\Om)}=||\rho^0||_{L^\infty(\Om)}$.
\end{proof}

\begin{rem}
 The expression \eqref{traceRenouv} shows that in the case of a zero boundary data $f$, the trace $\gamma\rho$ has some extra regularity, namely it is $\C([0,T];L^1(\d\Om))$.
\end{rem}

\begin{cor}[Uniqueness]\label{propUnicite}
 If $\rho$ and $\rho'$ are two weak solutions of the problem \eqref{equationNonAutonome}, then $\rho=\rho'$ almost everywhere.
\end{cor}

\section{Construction of approximated solutions and application to the existence}

In this section, we build a weak solution to the equation \eqref{equationNonAutonome} which, in view of the previous considerations, can be achieved by building a couple $(\rho_1,\rho_2)$ of solutions to the equations \eqref{equationTrho1}-\eqref{equationTrho2} (recall proposition \ref{propRhoRho1Rho2}). We will achieve the existence by convergence of an approximation scheme to the problem \eqref{equationTrho1}-\eqref{equationTrho2} where the difficulty is restricted to the approximation of the boundary condition. Then we establish an error estimate in the case of more regular data. In order to avoid heavy notations, we forget about the tilda when referring to the problem \eqref{equationTrho1}-\eqref{equationTrho2}. We place ourselves in the case where $\Om=(1,b)^2$.

\subsection{Construction of approximated solutions of the problem \eqref{equationTrho1}-\eqref{equationTrho2}}

Let $0=t_0<...<t_k<..<t_{K+1}=T$ be a uniform subdivision of $[0,T]$ with $t_{k+1}-t_k=\delta t$. For the equation \eqref{equationTrho2}, we give ourself uniform subdivisions $1=x_1<...<x_l<...<x_{L+1}=b$ and $1=\tt_1<...<\tt_m<...<\tt_{L+1}=b$, with $x_{l+1}-x_l=\tt_{m+1}-\tt_m=\delta x$. The scheme for the equation \eqref{equationTrho2} is then given by :
\begin{eqnarray}\label{schemaRho1}\left\lbrace\begin{array}{ll}
\rho_2^0(l,m)=\frac{1}{(\delta x)^2}\int_{x_l}^{x_{l+1}}\int_{\tt_m}^{\tt_{m+1}}\rho^0(x,\tt)dxd\tt & 1\leq l,m\leq L \\
\rho_2^{k+1}(l,m)=\rho_2^{k}(l,m)    & 1\leq k \leq K,\;1\leq l,m\leq L                                               
\end{array}\right. .
\end{eqnarray}
That is, $\rho_2^k(l,m)=\rho_2^0(l,m)$ for all $k,l,m$.\\
For the discretization of the equation \eqref{equationTrho1}, for each $k$ let $0=\tau_0<...<\tau_i<...<\tau_k=t_k$ with $\tau_{i+1}-\tau_i=\delta t$. Let $\sig:\begin{array}{ccc}
[0,1] & \rightarrow & \d\Om \\
s     & \mapsto     & \sig(s)                                                                                                                                                                                                                                            
\end{array}
$ be a parametrization of $\d\Om$ with $|\sig'(s)|=1$ a.e., so that for $g\in L^1(\d\Om)$ we have $\int_{\d\Om}g(\sig)d\sig=\int_0^1g(\sig(s))ds$. Let $0=s_1<...<s_j<...<s_{M+1}=1$ be an uniform subdivision with $s_{j+1}-s_j=\delta \sig$. The scheme is given by 
\begin{eqnarray}\label{schemaRho2}\left\lbrace\begin{array}{ll}
\rho_1^0(0,j)=N_j B^0((\rho_2^0)_{l,m}) + f^0_j& 1\leq j \leq M \\
\rho_1^{k+1}(i,j)=\rho_1^k(i,j)         & 1\leq k \leq K,\;0\leq i\leq k,\;1\leq j\leq M \\
\rho_1^{k+1}(k+1,j)=N_j 	B^{k+1}(\rho_1^{k+1},\rho_2^{k+1}) + f^{k+1}_j & 1\leq j\leq M\end{array}\right.
\end{eqnarray}
with 
\begin{align*}
B^k(\rho_1^k,\rho_2^k) & = \sum_{i=1}^{k-1}\sum_{j=1}^{M}\beta_{i,j}^1 \rho_1^k(i,j)\delta t \delta \sig + \sum_{l,m=1}^{L}\beta_{l,m}^2\rho_2^k(l,m)\left(\delta x\right)^2  \\
  & \simeq \int_0^{t_k}\int_{\d\Om}\beta(X(t_k;\tau,\sig))\rho_1(t_k,\tau,\sig)d\tau d\sig + \int_\Om\beta(X(t_k;0,Y)) \rho_2(t_k,Y)dY
\end{align*}
and 
\begin{equation}\label{approxDonnees}\begin{array}{c}
\beta_{i,j}^1:=\frac{1}{\delta t \delta \sig}\int_{\tau_i}^{\tau_{i+1}}\int_{\sig_j}^{\sig_{j+1}}\beta(X(t_k;\tau,\sig))d\sig d\tau,\;\beta_{l,m}^2:=\int_{x_l}^{x_{l+1}}\int_{\tt_m}^{\tt_{m+1}}\beta(X(t_k;0,(x,\tt)))dx d\tt \\
f^{k}_j:=\frac{1}{\delta t \delta\sig}\int_{t_k}^{t_{k+1}}\int_{\sig_j}^{\sig_{j+1}}f(t,\sig)d\sig dt,\; N_j:=\frac{1}{\delta\sig}\int_{\sig_j}^{\sig_{j+1}}N(\sig)d\sig .
\end{array}\end{equation}
Notice that the schemes \eqref{schemaRho1} and \eqref{schemaRho2} are well-posed since the definition of $\rho_1^{k+1}(k+1,j)$ involves values of $\rho_1^{k+1}(i,j)$ only with $0\leq i \leq k$.
We denote by $h=\delta t + \delta \sig + \delta x$ and define now the piecewise constant functions $\rho_{1,h}$ and $\rho_{2,h}$ on $\tQ_1$ and $[0,T[\times \bOm$ by, for $0\leq k\leq K,\;1\leq i \leq k,\;1\leq j\leq M$ and $1\leq l,m\leq L$
\begin{equation}\label{extensionRho1kRho2k}\begin{array}{ll}
   \rho_{1,h}(t,\tau,\sig(s))=\rho_1^k(i,j) & \text{ for }t\in[t_k,t_{k+1}[,\,\tau \in ]\tau_{i-1},\tau_i],\;s\in[s_j,s_{j+1}[\\ 
  \rho_{1,h}(t,\tau,\sig(s))=0             &\text{ for }t\in[t_k,t_{k+1}[,\,\tau \in ]t_k,t],\;s\in[s_j,s_{j+1}[ \\
  \rho_{2,h}(t,x,\tt)=\rho_2^k(l,m)        &\text{ for }t\in[t_k,t_{k+1}[,x\in[x_{l},x_{l+1}[,\;\tt\in[\tt_m,\tt_{m+1}[.
  \end{array}
\end{equation}
See the figure \ref{schemaDiscret} for an illustration.
\begin{figure}[!ht]
\begin{center}
 \input{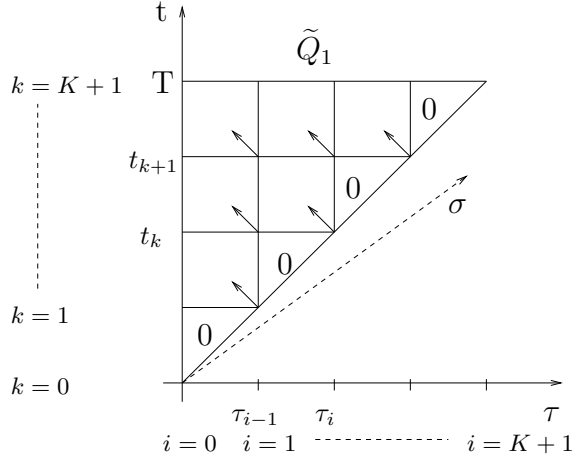}
\caption{Description of the discretization grid for $\tQ_1$, only in the $(\tau,t)$ plane. The arrows indicate the index used in assigning values to $\rho_{1,h}$ in each mesh (formula \eqref{extensionRho1kRho2k}).}\label{schemaDiscret}
\end{center}
\end{figure}
Notice that we have
\begin{equation}\label{exprNormL1}
\left|\left|\rho_{1,h}(t_k,\cdot)\right|\right|_{L^1(]0,t_k[\times\d\Om)}=\sum_{i=1}^k\sum_{j=1}^{M}\left|\rho_1^k(i,j)\right|\delta t \delta \sig,\;||\rho_{2,h}(t_k,\cdot)||_{L^1(\Om)}=\sum_{l,m=1}^M\abs{\rho_2^k(l,m)}(\delta x)^2.
\end{equation}

\begin{rem}~\\
\espace$\bullet$ We take the same discretization step in $x$ and $\tt$ for $\rho_2$ but  it would work the same with two different steps.\\
\espace$\bullet$ For more regular data, we could take point values instead of \eqref{approxDonnees}.\\
\espace$\bullet$ It will be clear from the following that the scheme would converge the same regardless to the value that we give to $\rho_1^0(0,j)$.\\
\end{rem}

\subsection{Discrete \textit{a priori} estimates}

We prove the equivalent of the proposition \ref{propEstimCont} in the discrete case. Notice that there exists a constant $C_{\sig}$ such that $\sum_{j=1}^{M} N_j \delta \sig \leq \int_{\d\Om} N(\sig)d\sig
+ C_\sig \delta \sig = 1+C_\sig \delta \sig:=||N||_h$ and $\sum_{j=1}^{M} f^{k+1}_j \delta \sig \leq ||f||_{L^\infty(]0,T[;L^1(\d\Om))}
+ C_\sig \delta \sig:=||f||_h$.

\begin{prop}[Discrete \textit{a priori} estimates]\label{propEstimDisc}
 Let $\left(\rho_1^k(i,j)\right)_{\scriptscriptstyle k,i,j}$ and $\left(\rho_2^k(l,m)\right)_{\scriptscriptstyle k,l,m}$ being given by \eqref{schemaRho1} and \eqref{schemaRho2} respectively. Then for all $k$
$$||\rho_{2,h}(t_k,\cdot)||_{L^1(\Om)}=||\rho^0||_{L^1(\Om)},\quad ||\rho_{2,h}||_{L^\infty(]0,T[\times\Om)}=||\rho^0||_{L^\infty(\Om)}$$
\begin{equation}\label{estimationL1Disc}
 ||\rho_{1,h}(t_k,\cdot)||_{L^1(]0,t_k[\times\d\Om)} \leq e^{t_k ||\beta||_{L^\infty}||N||_h}\left\lbrace||\rho^0||_{L^1(\Om)} + \frac{||f||_h}{||\beta||_{L^\infty}||N||_h}\right\rbrace,
\end{equation}
\begin{equation}\label{estimationLinfDisc}
 ||\rho_{1,h}||_{L^\infty(\tQ_1)}\leq ||N||_{L^\infty}||\beta||_{L^\infty}\underset{k}{\max}\left(||\rho_{1,h}(t_k,\cdot)||_{L^1}+||\rho^0||_{L^1}\right) +||f||_{L^\infty}.
\end{equation}
Moreover, if $\rho^0 \geq 0$ then $\rho_1^k(i,j),\;\rho_2^k(l,m)\geq 0$ for all $k,i,j,l,m$.
\end{prop}

\begin{proof}
The non-negativity of the scheme is straightforward from the definition. The estimate for $\rho_{2,h}$ follows directly from the scheme \eqref{schemaRho2}. For the $L^1$ estimate on $\rho_{1,h}$ we compute, using the scheme \eqref{schemaRho1}
\begin{align*}
||\rho_{1,h}(t_{k+1},\cdot)||_{L^1(]0,t_{k+1}[\times\d\Om)} & =\sum_{i=1}^{k+1}\sum_{j=1}^{M}\left|\rho^{k+1}(i,j)\right|\delta t \delta \sig \\
 & = \sum_{i=1}^k\sum_{j=1}^{M}\left|\rho^k(i,j)\right| \delta t \delta \sig +\left| B^{k+1}(\rho_1^{k+1},\rho_2^{k+1})\right|\delta t\sum_{j=1}^{M}N_j \delta\sig + \delta t \sum_{j=1}^{M}\left|f^{k+1}_j\right| \delta\sig\\
 & \leq ||\rho_{1,h}(t_k)||_{L^1(]0,t_k[\times\d\Om)} + \left|B^{k+1}(\rho_1^{k+1},\rho_2^{k+1})\right|\delta t||N||_h + \delta t||f||_h
\end{align*}
Now from the expression of $B^{k+1}(\rho_1^{k+1},\rho_2^{k+1})$
\begin{align*}
\left|B^{k+1}(\rho_1^{k+1},\rho_2^{k+1})\right| \leq ||\beta||_{L^\infty}||\rho_{1,h}(t_k,\cdot)||_{L^1} +||\beta||_{L^\infty} ||\rho_{2,h}(t_k,\cdot)||_{L^1}.
\end{align*}
Thus we obtain
\begin{align*}
||\rho_{1,h}(t_{k+1},\cdot) ||_{L^1} & \leq \left(1+||\beta||_{L^\infty}\delta t||N||_h\right)||\rho_{1,h}(t_{k},\cdot)||_{L^1} + ||\beta||_{L^\infty}\delta t||N||_h||\rho_{2,h}(t_k,\cdot)||_{L^1} + \delta t||f||_h
\end{align*}
Now using a discrete Gronwall lemma we obtain
$$||\rho_{1,h}(t_{k+1},\cdot) ||_{L^1} \leq e^{||\beta||_{L^\infty}||N||_h t_k}\left\lbrace ||\rho_{1,h}(t_0,\cdot)||_{L^1} + \frac{||\beta||_{L^\infty}||N||_h||\rho_{2,h}(t_k,\cdot)||_{L^1} + ||f||_h}{||\beta||_{L^\infty}||N||_h} \right\rbrace$$
Using $||\rho_{1,h}(t_0,\cdot)||_{L^1}=0$ and $||\rho_{2,h}(t_k,\cdot)||_{L^1(\Om)}= ||\rho^0||_{L^1(\Om)}$ ends the proof of the $L^1$ estimate. \\
For the $L^\infty$ estimate, we remark that 
\begin{align*}
||\rho_{1,h}||_{L^\infty(\tQ_1)} & =\underset{k}{\max}\,\underset{i,j}{\max}|\rho_1^k(i,j)|=\underset{k}{\max}\,\underset{j}{\max}\left(\left|B^k(\rho_1^k,\rho_2^k) N_j + f^k_j\right|\right)\leq ||N||_{L^\infty}\underset{k}{\max}\left|B^k(\rho_1^k,\rho_2^k)\right| + ||f||_{L^\infty} \\
  & \leq ||N||_{L^\infty}||\beta||_{L^\infty}\underset{k}{\max}\left(||\rho_{1,h}(t_k,\cdot)||_{L^1}+||\rho_{2,h}(t_k,\cdot)||_{L^1}\right) +||f||_{L^\infty}.
\end{align*}
\end{proof}
\subsection{Application to existence of solutions to the continuous problem \eqref{equationTrho1}-\eqref{equationTrho2}}
\begin{thr}[Existence]\label{existence}
Under the assumptions \eqref{hypDonnees}, there exists $\rho_1 \in L^\infty(\tQ_1)$ and $\rho_2\in L^\infty(]0,T[\times\Om)$ such that $\rho_{1,h}\underset{h\rightarrow 0}{\rightharpoonup} \rho_1$ and $\rho_{2,h}\underset{h\rightarrow 0}{\rightharpoonup} \rho_2$ for the weak-$\ast$ topology of $L^\infty$. Furthermore, $(\rho_1,\rho_2)$ is the unique weak solution of \eqref{equationTrho1}-\eqref{equationTrho2}.
\end{thr}
\begin{proof}
Uniqueness of the solution is straightforward for the problem \eqref{equationTrho2} and follows from the $L^1$ estimate on $\rho_1$ which can be derived following the proof of the proposition \ref{propEstimCont}. The proof for the existence is rather classical and consists in passing to the limit in discrete weak formulations of \eqref{equationTrho1} and \eqref{equationTrho2}.
 From the previous proposition, we obtain that the families $\left\lbrace\rho_{1,h}\right\rbrace_{\delta t,\,\delta \sig}$ and $\left\lbrace\rho_{2,h}\right\rbrace_{\delta t,\,\delta x}$ are bounded in $L^\infty$ and thus there exist $\rho_1 \in L^\infty(\tQ_1)$, $\rho_2\in L^\infty(]0,T[\times\Om)$ and some subsequences $\rho_{1,h_n}$ and $\rho_{2,h_n}$ such that $\rho_{1,h_n}\underset{h_n\rightarrow 0}{\rightharpoonup} \rho_1$ and $\rho_{2,h_n}\underset{h_n\rightarrow 0}{\rightharpoonup} \rho_2$ for the weak-$\ast$ topology of $L^\infty$. We have to prove now that $(\rho_1,\rho_2)$ is a weak solution of \eqref{equationTrho1}-\eqref{equationTrho2}. The uniqueness of solutions to the equation implies then by standard argument that the whole sequence converges. It remains to prove that $(\rho_1,\rho_2)$ solves \eqref{equationTrho1}-\eqref{equationTrho2}. \\
\espace $\bullet$\textit{The function $\rho_2$ is a weak solution of \eqref{equationTrho2}.} Let $\phi_2$ be a test function for \eqref{equationTrho2}. We have
\begin{align*}\int_0^T\int_\Om \rho_{2,h_n}(t,Y)\d_t\phi_2(t,Y)dYdt& =\sum_{k=0}^{K}\sum_{l,m=1}^{L} \rho_2^k(l,m)\int_{t_k}^{t_{k+1}}\int_{x_l}^{x_{l+1}}\int_{\tt_m}^{\tt_{m+1}}\d_t \phi_2(t,x,\tt)d\tt dx dt \\
 & = \sum_{k=0}^{K}\sum_{l,m=1}^{L}\rho_2^k(l,m) \Phi_2(t_{k+1},l,m)(\delta x)^2-\sum_{k=0}^{K}\sum_{l,m=1}^{L}\rho_2^k(l,m) \Phi_2(t_{k},l,m)(\delta x)^2
\end{align*}
where we denoted $\Phi_2(t_{k},l,m):=\frac{1}{(\delta x )^2}\int_{x_l}^{x_{l+1}}\int_{\tt_m}^{\tt_{m+1}} \phi_2(t_k,x,\tt)d\tt dx$. Using the scheme ($\rho_2^k(l,m)$ is constant in $k$) and $\Phi_2(t_{K+1},l,m)=0$ since $t_{K+1}=T$, we obtain
\begin{align*}
\int_0^T\int_\Om & \rho_{2,h_n}(t,Y)\d_t\phi_2(t,Y)dYdt =\sum_{l,m=1}^{L}\rho^{K}_2(l,m)\Phi_2(T,l,m)(\delta x)^2 -  \sum_{l,m=1}^{L}\rho_2^0(l,m)\Phi_2(0,l,m)(\delta x)^2 \\
 & = -\sum_{l,m=1}^{L}\rho_2^0(l,m)\Phi_2(0,l,m)(\delta x)^2 = -\int_\Om \rho_{2,h_n}^0(Y)\phi(0,Y) \xrightarrow[h_n\rightarrow 0]{} - \int_\Om \rho^0(Y)\phi(0,Y)dY
\end{align*}
since $\rho_{2,h_n}^0\xrightarrow[h_n\rightarrow 0]{L^1}\rho^0$. Observing that the left hand side converges to $\int_0^T\int_\Om\rho_2 \d_t \phi_2(t,Y)dY dt$ gives the result. \\
\espace $\bullet$\textit{The function $\rho_1$ is a weak solution of \eqref{equationTrho1}.} Let $\phi_1$ be a test function for \eqref{equationTrho1}. Then the same calculation as above shows, with $\Phi_1(t_k,i,j):=\frac{1}{\delta t \delta \sig}\int_{\tau_{i-1}}^{\tau_i}\int_{\sig_j}^{\sig_{j+1}}\phi_1(t_k,\tau,\sig)d\sig d\tau$ and using that $\Phi_1(t_{K+1},i,j)=0$ as well as $\rho_1^{k+1}(i,j)=\rho_1^k(i,j)$ for $1\leq i \leq k$ and $1\leq j \leq M$
\begin{align*}
&\int_{\tQ_1} \rho_{1,h_n}(t,\tau,\sig) \d_t\phi_1(t,\tau,\sig)  d\sig d\tau dt = \sum_{k=1}^K\sum_{i=1}^k\sum_{j=1}^M \rho_1^k(i,j)\Phi_1(t_{k+1},i,j)\delta t \delta \sig - \sum_{k=1}^K\sum_{i=1}^k\sum_{j=1}^M \rho_1^k(i,j)\Phi_1(t_{k},i,j)\delta t \delta \sig \\
& = \sum_{i=1}^K\sum_{j=1}^M\rho_1^K(i,j)\Phi_1(t_{K+1},i,j)\delta t \delta \sig + \sum_{k=1}^{K-1}\sum_{i=1}^k\sum_{j=1}^M\rho_1^k(i,j)\Phi_1(t_{k+1},i,j)\delta t \delta \sig \\
& -\sum_{k=1}^{K-1}\sum_{i=1}^{k+1}\sum_{j=1}^M\rho_1^{k+1}(i,j)\Phi_1(t_{k+1},i,j)\delta t \delta \sig -\sum_{j=1}^M\rho_1^1(1,j)\Phi_1(t_1,1,j)\delta t \delta \sig \\
& = - \sum_{k=1}^{K-1}\sum_{j=1}^M\rho_1^{k+1}(k+1,j)\Phi_1(t_{k+1},k+1,j)\delta t \delta \sig -\sum_{j=1}^M\rho_1^1(1,j)\Phi_1(t_1,1,j)\delta t \delta \sig \\
& = - \sum_{k=1}^{K}\sum_{j=1}^M\left(N_j B^{k}(\rho_1^{k},\rho_2^{k})+f^{k}_j\right)\Phi_1(t_{k},k,j)\delta t \delta \sig
\end{align*}
Defining the following piecewise constant functions : $B_h(t,\rho_{1,h},\rho_{2,h})=B^k(\rho_1^k,\rho_2^k),\;N_h(\sig(s))=N_j,\;f_h(t,\sig(s))=f^k_j$ and $\Phi_{1,h}(t,\sig(s))=\Phi_1(t_k,k,j)$ on $[t_k,t_{k+1}[\times[s_j,s_{j+1}[$, the previous equality reads
$$\int_{\tQ_1} \rho_{1,h_n}(t,\tau,\sig) \d_t\phi_1(t,\tau,\sig)  d\sig d\tau dt=\int_{\delta t}^T\int_{\d\Om}(B_{h_n}(t,\rho_{1,h_n},\rho_{2,h_n}) N_{h_n}(\sig) +f_{h_n}(t,\sig))\Phi_{1,h_n}(t,\sig)d\sig dt.$$
We need the following lemma in order to conclude.
\begin{lem}
 We have 
$$B_{h_n}(t,\rho_{1,h_n},\rho_{2,h_n})\underset{h_n\rightarrow 0}{\rightharpoonup}\tB(t,\rho_1,\rho_2)\;\ast-L^\infty(]0,T[).$$
\end{lem}
\begin{proof}
We define the piecewise constant function $\beta^1_h(\tau,\sig)$ as for $N_h$ and $f_h$ and $\beta^2_h(X)=\beta^2_{l,m}$ for $X\in [x_{l},x_{l+1}[\times[\tt_{m},\tt_{m+1}[$. Let $t\in[t_k,t_{k+1}[$, then
$$B_h(t,\rho_{1,h},\rho_{2,h})=B^k(\rho_1^k,\rho_2^k) =\int_0^t\int_{\d\Om}\beta^1_h(\tau,\sig)\rho_{1,h}(t,\tau,\sig)d\tau d\sig - \sum_{j=1}^M\beta^1_{k,j}\rho_1^k(k,j)\delta t \delta \sig + \sum_{l,m=1}^L\beta^2_{l,m}\rho_2^k(l,m)(\delta x)^2$$
since we defined $\rho_h(t,\tau,\sig)=0$ for $\tau\in]t_k,t]$. Thus, for $\psi \in L^1(]0,T[)$ we have
\begin{align*}
\int_0^T B_h(t,\rho_{1,h},\rho_{2,h})\psi(t)dt & =\int_0^T\int_0^t\int_{\d\Om} \beta^1_h(\tau,\sig)\rho_{1,h}(t,\tau,\sig)\psi(t)d\sig d\tau dt - \delta t\sum_{k=0}^{K}\sum_{j=1}^M\beta^1_{k,j} \rho_1^k(k,j)\int_{t_k}^{t_{k+1}}\psi(t)dt \delta \sig \\
 & + \int_0^T\int_\Om \beta^2_h(X) \rho_{2,h}(t,X) \psi(t) dX dt
\end{align*}
and we obtain the result by using $\rho_{1,h_n}\underset{h_n\rightarrow 0}{\rightharpoonup}\rho_1\;\ast-L^\infty$, $\rho_{2,h_n}\underset{h_n\rightarrow 0}{\rightharpoonup}\rho_2\;\ast-L^\infty$, $\beta_{h_n}\xrightarrow[h_n\rightarrow 0]{L^1}\beta$, $||\beta_{h_n}||_{L^\infty}\leq C$ and noticing that the second term goes to zero in view of the $L^\infty$ bounds on $\rho_{1,h}$ (proposition \ref{propEstimDisc}) and $\beta$.
\end{proof}
Using the lemma as well as $N_{h_n},\,f_{h_n}\underset{h_n \rightarrow 0}{\rightharpoonup} N,f\;\ast-L^\infty$, $||N_{h_n}||_{L^\infty}\leq C$ and $\Phi_{1,h_n}\xrightarrow[h_n\rightarrow 0]{\C([0,T]\times\d\Om)}\phi(t,t,\sig)$, the previous calculations give
$$\int_{\tQ_1} \rho_{1,h_n}(t,\tau,\sig)\d_t\phi_1(t,\tau,\sig) d\sig d\tau dt\xrightarrow[h\rightarrow 0]{}-\int_0^T\int_{\d\Om}\left\lbrace N(\sig)\tB(t,\rho_1,\rho_2)+f(t,\sig)\right\rbrace \phi(t,t,\sig)d\sig dt.$$
On the other hand the left hand side also goes to $\int_{\tQ_1}\trho_1(t,\tau,\sig)\d_t \phi_1(t,\tau,\sig)d\sig d\tau dt$. This proves that $\rho_1$ verifies the definition \ref{defSolRho1Rho2} and ends the proof.
\end{proof}

\subsection{Error estimate}

We establish now an error estimate for the approximation of the equations \eqref{equationTrho1}-\eqref{equationTrho2}. For this section, we make the following assumptions on the data : 
\begin{equation}\label{hypDonneesReg}
 \rho^0 \in \Winf(\Om),\;\beta\in \Winf(\Om),\;N\in \Winf(\d\Om),\,N\geq 0,\,\int_{\d\Om}N(\sig)d\sig=1,\; f\in \Winf(]0,T[\times \d\Om)
\end{equation}
It can be noticed that in order to perform the weak convergence of the approximated solutions and establish theoretical existence to the continuous problem, we did not need to approximate the characteristics $X(t;\tau,\sig)$ of the equation. In view of the error estimate though, we need to use another approximation of the data than \eqref{approxDonnees}. For $\beta(X(t;\tau,\sig))$ we have to introduce an approximation $X_h(t;\tau,\sig)$ of the characteristics given by a numerical integrator of the ODE system \eqref{g1}-\eqref{g2}. Then we define
\begin{equation}\label{approxDonneesCont}
\begin{array}{c}
   \beta^1_{i,j}:=\beta(X_h(t_k;\tau_i,\sig_j)),\; \beta^2_{l,m}:=\beta(X_h(t_k;0,(x_l,\tt_m)) \\
  f_j^k:=f(t_k,\sig_j),\; N_j:=N(\sig_j).
  \end{array}
\end{equation}
For $g_1$ and $g_2$ being two continuous functions on $\tQ_1$ and $]0,T[\times\Om$ respectively, we define 
$$\begin{array}{ll}
\P_1 g_1(t,\tau,\sig(s))=g_1(t_k,\tau_i,\sig_j) &\text{ for }t\in[t_k,t_{k+1}[,\,\tau\in]\tau_{i-1},\tau_i],\;s \in [s_j,s_{j+1}[\\
  \P_1 g_1(t,\tau,\sig(s))=0  &\text{ for }t\in[t_k,t_{k+1}[,\,\tau \in ]t_k,t],\;s\in[s_j,s_{j+1}[ \\
  \P_2 g_2(t,x,\tt)=g_2(t_k,x_l,\tt_m)   & \text{ for }t\in[t_k,t_{k+1}[,\,x\in]x_{l},x_{l+1}],\;\tt \in [\tt_m,\tt_{m+1}[
  \end{array}.
$$

\begin{lem}[Projection error]\label{lemProjection}
Let $(g_1,g_2)\in \Winf(\tQ_1)\times\Winf(]0,T[\times\Om)$. Then there exists $C_{\P_1}$ and $C_{\P_2}$ such that
\begin{eqnarray}\begin{array}{c}
||g_1(t_k,\cdot) - \P_1 g_1(t_k,\cdot)||_{L^\infty(]0,t_k[)}\leq C_{\P_1} h,\quad
||g_2(t_k,\cdot) - \P_2 g_2(t_k,\cdot)||_{L^\infty(\Om)}\leq C_{\P_2} h.
\end{array}\end{eqnarray}
\end{lem}
We don't give the proof of this lemma since it is classical. We define $e_{1,h}:=\rho_{1,h}-\P_1\trho_1$ and $e_{2,h}:=\rho_{2,h}-\P_2\trho_2$ the errors of the schemes, with $(\trho_1,\trho_2)$ solving the problem \eqref{equationTrho1}-\eqref{equationTrho2}. From the equation \eqref{equationTrho1} we have
\begin{equation*}
 \left\lbrace\begin{array}{l}
  \trho_1(t_{k+1},\tau_i,\sig_j)=\trho_1(t_{k},\tau_i,\sig_j),\quad 0\leq k\leq K,\,0\leq i\leq k,\,1\leq j \leq M \\
  \trho_1(t_{k+1},\tau_{k+1},\sig_j)=N(\sig_j)\tB(t_{k+1},\trho_1,\trho_2) + f(t_k,\sig_j)
 \end{array}\right.
\end{equation*}
and thus, subtracting this to \eqref{schemaRho1} and denoting $e_1^k(i,j)=e_{1,h}(t_k,\tau_i,\sig_j)$ we obtain
\begin{equation}\label{schemaErreur}
 \left\lbrace\begin{array}{l}
  e_1^{k+1}(i,j)=e_1^k(i,j),\quad 0\leq k\leq K,\,0\leq i\leq k,\,1\leq j \leq M \\
  e_1^{k+1}(k+1,j)=N_j E^{k+1} + r_j^{k+1}
 \end{array}\right.
\end{equation}
with 
\begin{align*}
E^{k+1} & =\sum_{i=1}^{k}\sum_{j=1}^{M}\beta_{i,j}^1e_1^{k+1}(i,j)\delta t\delta \sig + \sum_{l,m=1}^{L} \beta_{l,m}^2e_2^{k+1}(l,m)(\delta x)^2 \\
r_j^{k+1} & =N_j\left(B^{k+1}\left(\left(\trho_1(t_{k+1},\tau_i,\sig_j)\right)_{\scriptscriptstyle i,j},\left(\trho_2(t_{k+1},x_l,\tt_m)\right)_{\scriptscriptstyle l,m}\right) - \tB(t_{k+1},\trho_1,\trho_2)\right).
\end{align*}
Hence the truncation error of the scheme $r_j^{k+1}$ comes only from the quadrature error coming from the approximation of the integral in $\tB(t_k,\trho_1,\trho_2)$. 
\begin{lem}[Truncation error]
Assume \eqref{hypDonneesReg} and that $(\beta\circ X_1)\trho_1 \in \Winf(\tQ_1),\; (\beta\circ X_2)\trho_2\in\Winf(]0,T[\times\Om)$. Then there exists $C_r$ such that 
$$\underset{k,j}{\max}\;|r^k_j|\leq C_r h.$$
\end{lem}
\begin{proof}
 We have 
\begin{align*}r_j^k & =N_j[\sum_{i=1}^{k-1}\sum_{j=1}^{M}\left(\beta_{i,j}^1-\beta(X_1(t_k;\tau_i,\sig_j))\right) \trho_1(t_k,\tau_i,\sig_j)\delta t \delta \sig + \sum_{l,m=1}^{L}\left(\beta_{l,m}^2-\beta(X_2(t_k;x_l,\tt_m)) \right)\trho_2(t_k,x_l,\tt_m)\left(\delta x\right)^2 \\
 & + \sum_{i=1}^{k-1}\sum_{j=1}^{M}\beta(X_1(t_k;\tau_i,\sig_j)) \trho_1(t_k,\tau,\sig)\delta t \delta \sig + \sum_{l,m=1}^{L}\beta(X_2(t_k;x_l,\tt_m))\trho_2(t_k,x_l,\tt_m)\left(\delta x\right)^2 \\
 & - \int_0^{t_{k-1}}\int_{\d\Om}\beta(X_1(t_k;\tau,\sig))\trho_1(t_k,\tau,\sig)d\tau d\sig  - \int_\Om \beta(X_2(t_k,Y))\trho_2(t_k,Y)dY \\
& - \int_{t_{k-1}}^{t_k}\int_{\d\Om}\beta(X_1(t_k;\tau,\sig))\trho_1(t_k,\tau,\sig)d\tau d\sig].
\end{align*}
Thus
\begin{align*}
 \left|r_j^k\right| \leq & ||N||_{L^\infty}\{
||\beta||_{W^{1,\infty}}\left(||X_{1,h}-\P_1 X_1||_{L^\infty}||\P_1\trho_1||_{L^1} + ||X_{2,h}-\P_2 X_2||_{L^\infty}||\P_2\trho_2||_{L^1}\right) \\
& + \sum_{i=1}^{k-1}\sum_{j=1}^M\int_{\tau_{i-1}}^{\tau_i}\int_{\sig_j}^{\sig_{j+1}}\left| \P_1\left[\left(\beta\circ X_1\right)\trho_1\right](t_k,\tau,\sig) -\left(\beta\circ X_1\right)\trho_1(t_k,\tau,\sig)\right|d\tau d\sig \\ 
& + \sum_{l,m=1}^{L}\int_{x_l}^{x_{l+1}}\int_{x_m}^{x_{m+1}}\left| \P_2\left[\left(\beta\circ X_2\right)\trho_2\right](t_k,Y) -\left(\beta\circ X_2\right)\trho_2(t_k,x,\tt)\right|dx d\tt + \left|\left|\left(\beta\circ X_1\right)\trho_1\right|\right|_{L^\infty} h\}.
\end{align*}
Using the lemma \ref{lemProjection} and the $L^1$ a priori estimate of proposition \ref{propEstimCont} gives the result.
\end{proof}
\begin{rem}[Order of the truncation error]\label{remErrTronc}
 In order to have a better order for the truncation error we could use a more sophisticated quadrature method like for instance the trapezoid method on $\Om$ for $\trho_2$ and on $[0,t_{k-1}[\times \d\Om$ for $\trho_1$ (completed by a left rectangle method on $[t_{k-1}, t_k[\times \d\Om$). Adapting the previous proof shows that if the numerical integrator used for the characteristics has order larger than $2$, then the truncation error would have order $2$ (order of the trapezoid method).
\end{rem}

\begin{prop}[Error estimate]\label{propErrTrho1Trho2}
Assume \eqref{hypDonneesReg} and that $(\trho_1,\trho_2)\in\Winf(\tQ_1)\times\Winf(]0,T[\times \Om)$ is a regular solution of \eqref{equationTrho1}-\eqref{equationTrho2}. Let $\rho_{1,h}$ and $\rho_{2,h}$ solve \eqref{schemaRho1} and \eqref{schemaRho2}. Then there exists some constants $\tilda{C}_1$ and $\tilda{C}_2$ such that
\begin{eqnarray}\begin{array}{c}
||\rho_{1,h}(t_k,\cdot)-\trho_1(t_k,\cdot)||_{L^1(]0,t_k[)}\leq \tilda{C}_1 h,\quad
||\rho_{2,h}(t_k,\cdot)-\trho_2(t_k,\cdot)||_{L^1(\Om)}\leq \tilda{C}_2 h
\end{array}\end{eqnarray}
\end{prop}

\begin{proof}In view of the lemma \ref{lemProjection}, it is sufficient to prove the proposition with $\P_s\trho_s(t_k,\cdot)$ instead of $\trho_s(t_k,\cdot)$ (with $s=1,2$). For the second estimate, we notice that 
$$||\rho_{2,h}(t_k,\cdot)-\P_2\trho_2(t_k,\cdot)||_{L^1(\Om)}=||e_{2,h}(t_k,\cdot)||_{L^1(\Om)}=\sum_{l,m}\left|e_2^k(l,m)\right|(\delta x)^2=\sum_{l,m}\left|\rho^0_2(l,m) - \rho^0(x_l,\tt_m)\right|(\delta x)^2$$
and the result follows from the definition of $\rho^0_2(l,m)$. For the first one, we have
 $$||\rho_{1,h}(t_k,\cdot)-\P_1\trho_1(t_k,\cdot)||_{L^1(]0,t_k[)}=||e_{1,h}(t_k,\cdot)||_{L^1(]0,t_k[)}=\sum_{i=1}^k\sum_{j=1}^{M}\left|e_1^k(i,j)\right|\delta t \delta \sig.$$
We can compute, using \eqref{schemaErreur}
\begin{align*}
||e_{1,h}(t_{k+1},\cdot)||_{L^1} & \leq \sum_{i=1}^k\sum_{j=1}^{M}\left|e_1^{k+1}(i,j)\right|\delta t \delta \sig +  \left| E^{k+1} \right|\delta t \sum_{j=1}^{M} N_j \delta \sig  + \delta t \sum_{j=1}^M \left| r_j^{k+1} \right|\delta \sig \\
  & \leq ||e_{1,h}(t_k,\cdot)||_{L^1} + \delta t||\beta||_{\infty}||N||_h\left\lbrace||e_{1,h}(t_{k},\cdot)||_{L^1} + ||e_{2,h}(t_{k+1},\cdot)||_{L^1} \right\rbrace + C_r h \delta t \\
 & \leq (1+\delta t||\beta||_{\infty}||N||_h)||e_{1,h}(t_k,\cdot)||_{L^1} + \tilda{C}_2||\beta||_{\infty}||N||_h h^p \delta t + C_r h\delta t
\end{align*}
and conclude using a discrete Gronwall lemma.
\end{proof}

\begin{rem}[Order of the error]~\\
 \espace $\bullet$ By looking more carefully at the propagation of errors in the proof, we see that if we set $\rho_2^0(l,m)=\rho^0(l,m)$ (which is valid under \eqref{hypDonneesReg}), the error on $\trho_2$ comes only from the projection error. \\
\espace $\bullet$ If in addition, we follow the remark \ref{remErrTronc} for the approximation of the data, then the error between  $\rho_{1,h}$ and $\P_1\trho_1$ would be of order $2$ if we had used a trapezoid method for the integral term in $\tB(t_k,\trho_1,\trho_2)$.
\end{rem}

\subsection{Application to the approximation of the problem \eqref{equationNonAutonome}}
We explain now how we approximate the solution of \eqref{equationNonAutonome} from the approximation of the solutions of equations \eqref{equationTrho1}-\eqref{equationTrho2} given by the schemes \eqref{schemaRho1}, \eqref{schemaRho2}.
We translate formula \eqref{rhoRho1Rho2} at the discrete level thanks to $\trho_{1,h}$, $\trho_{2,h}$ given by \eqref{extensionRho1kRho2k} and the solutions $\trho_1^k(i,j)$, $\trho_2^k(i,j)$ of the schemes \eqref{schemaRho1} and \eqref{schemaRho2} :
\begin{align}\label{defRhoh}
\rho_{h}(t,X) & :=\underset{:=\rho_{1,h}}{\underbrace{\trho_{1,h}(t,\tau^t(X),\sig^t(X))J_{1,h}^{-1}(t,\tau^t(X),\sig^t(X))\mathbf{1}_{X\in\Om_1^t}}} + \underset{:=\rho_{2,h}}{\underbrace{\trho_{2,h}(t,Y(X))J_{2,h}^{-1}(t,Y(X))\mathbf{1}_{X\in\Om_2^t}}}.
\end{align}
The jacobians of the changes of variables $J_1(t;\tau,\sig)=|G(\tau,\sig)\cdot\nuu (\sig)|e^{\int_\tau^t \div \,
G(u,X(u;\tau,\sig))du}\text{ and }J_2(t;Y)=e^{\int_0^t \div \,
G(u,X(u;0,Y))du}$ are approximated respectively by $J_{1,h}$ and $J_{2,h}$, piecewise constant functions constructed similarly as in \eqref{extensionRho1kRho2k} through $J_{1}^k(i,j):=e^{\T_1(k,i,j)}$ and $J_{2}^k(l,m):=e^{\T_2(k,l,m)}$, where $\T_1$ and $\T_2$ are one-dimensional quadrature methods such that $\T_1(k,i,j)\simeq \int_{\tau_i}^{t_k}\div G(X(s;\tau_i,\sig_j))ds $ and $\T_2(k,l,m)\simeq\int_0^{t_k}\div G(X(s;0,(x_l,\tt_m)))ds$. The errors of these quadrature methods are denoted by $r_1$, $r_2$ and are assumed to be of order $\alpha_1$, $\alpha_2$ :
$$\barre{r_1}:=\underset{k,i,j}{\max}\,\left|r_1(k,i,j)\right|\leq C_{q}(\delta t)^{\alpha_1},\;\barre{r_2}:=\underset{k,l,m}{\max}\,\left|r_2(k,l,m)\right|\leq C_{q}(\delta t)^{\alpha_2}.$$
Hence, we have
\begin{equation}\label{defJacApp}
 J_{1}^k(i,j)=J_1(t_k,\tau_i,\sig_j)e^{-r_1(k,i,j)},\;J_{2}^k(l,m):=e^{\T_2(k,l,m)}=J_2(t_k,x_l,\tt_m)e^{-r_2(k,l,m)}.
\end{equation}
We define the following meshes :
$$\begin{array}{c}
V_1(k,i,j)=\{(t,X(t;\tau,\sig(s)));\;t\in[t_k,t_{k+1}[,\,\tau\in]\tau_{i-1},\tau_i],\,s\in[s_j,s_{j+1}[\} \\
V_2(k,l,m)=\{(t,X(t;0,(x_l,\tt_m)));\; t \in[t_k,t_{k+1}[,\,x\in[x_l,x_{l+1}[,\,\tt\in[\tt_m,\tt_{m+1}[\}   
  \end{array}
$$
and, for a function $g\in\C([0,T]\times\bOm)$
\begin{equation}\label{defPg}
\P g(t,X)=g(t_k,X(t_k;\tau_i,\sig_j))\mathbf{1}_{(t,X)\in V_1(k,i,j)}+g(t_k,X(t_k;0,(x_l,\tt_m)))\mathbf{1}_{(t,X)\in V_2(k,l,m)}.
\end{equation}

\begin{rem}\label{remProj}
 In the same way as the lemma \ref{lemProjection}, there exists a constant $C_\P$ such that for all function $g\in \Winf(]0,T[\times \Om)$
$$||g-\P g||_{L^1(]0,T[\times \Om)}\leq C_\P h.$$
\end{rem}

\begin{thr}
 Suppose that $\rho\in\Winf(]0,T[\times\Om)$ is a regular solution of the equation \eqref{equationNonAutonome} and let $\rho_h$ be defined by \eqref{defRhoh}. Then there exists a constant $C$ such that
$$\underset{t\in[0,T]}{\sup}\,||\rho_h(t,\cdot)-\rho(t,\cdot)||_{L^1(\Om)}\leq C h.$$
\end{thr}

\begin{proof}
 In view of the remark \ref{remProj}, it is again sufficient to prove the proposition with $\P\rho$ instead of $\rho$. Let $t\in[t_k,t_{k+1}[$, then $||\rho_h(t,\cdot)-\P\rho(t,\cdot)||_{L^1(\Om)}=||\rho_{1,h}(t_k,\cdot)-\P\rho_1(t_k,\cdot)||_{L^1(\Om_1^{t_k})}+||\rho_{2,h}(t_k,\cdot)-\P\rho_2(t_k,\cdot)||_{L^1(\Om_2^{t_k})}$ with $\rho_s(t,X):=\rho(t,X)\mathbf{1}_{X\in\Om_s^t}$ ($s=1,2$). We do the proof only for $\rho_1$ since it is similar for $\rho_2$. We also don't write the dependency in $\sig$ in order to avoid heavy notations. To obtain the complete proof it suffices to add integrals with respect to $\sigma$ in the following and $\sig$ in all the functions. Doing the change of variables $X_1$ we have, noticing that $\P\rho_1(t_k,X(t_k;\tau))=\P_1\trho_1(t_k,\tau)\P_1 J_1^{-1}(t_k,\tau)$
\begin{align*}
||\rho_{1,h}(t_k,\cdot) & -\P\rho_1(t_k,\cdot)||_{L^1(\Om_1^{t_k})}=\int_0^{t_k}\left|\trho_{1,h}(t_k,\tau)J_{1,h}^{-1}(t_k,\tau) - \P_1\trho_1(t_k,\tau)\P_1 J_1^{-1}(t_k,\tau)\right|J_1(t_k,\tau)d\tau \\
 & \leq \int_0^{t_k}\left|\trho_{1,h}(t_k,\tau)\right|\left|J_{1,h}^{-1}(t_k,\tau)J_1(t_k,\tau) - 1\right| d\tau + \int_0^{t_k}\left|\trho_{1,h}(t_k,\tau) - \P_1\trho_1(t_k,\tau)\right| d\tau + \\
 & + \int_0^{t_k}\left|\P_1\trho_1(t_k,\tau)\right|\left|1-\P_1 J_1^{-1}(t_k,\tau)J_1(t_k,\tau)\right| d\tau.
\end{align*}
Now we have, using the definition \eqref{defJacApp}
$$\left|J_{1,h}^{-1}J_1 - 1\right|=\left|\P J_1^{-1}e^{-r_1}J_1 - 1\right|\leq \left|e^{\barre{r_1}}\right|\frac{1}{\abs{\P J_1}}\abs{J_1 - \P J_1} + \abs{e^{-r_1}-1}$$
Thus, since $\left|\left|\frac{1}{J_1}\right|\right|_{L^\infty}<\infty$ from formula \eqref{formule_jacobien} and the fact that $G\cdot \nu\geq m>0$, and using $\abs{e^{-r_1}-1}\leq 2\barre{r_1}$, there exists $C_J$ such that 
$$||J_{1,h}^{-1}(t_k,\tau)J_1(t_k,\tau) - 1||_{L^\infty}\leq C_J h,\text{ and }||1-\P_1 J_1^{-1}(t_k,\tau)J_1(t_k,\tau)||_{L^\infty}\leq C_J h .$$
The last inequality comes from the lemma \ref{lemProjection} since $J_1 \in \W^{1,\infty}$ from the formula \eqref{formule_jacobien}. Using then the continuous and discrete \textit{a priori} $L^1$ estimates and the proposition \ref{propErrTrho1Trho2} gives the result.
\end{proof}

\begin{rem}
 In the case of less regularity on the solution, we still have $\rho_{h}\underset{h\rightarrow 0}{\rightharpoonup}\rho,\;\ast-L^\infty(]0,T[\times\Om)$. Indeed, we write $\rho_h = \trho_{1,h}J_{1,h}^{-1} + \trho_{2,h}J_{2,h}^{-1}=\trho_{1,h}J_1^{-1} + \trho_{2,h}J_2^{-1} + \trho_{1,h}(J_{1,h}-J_1) + \trho_{2,h}(J_{2,h}-J_2)$. Then we use that for $s=1,2$ $J_{s,h}^{-1}\xrightarrow[h\rightarrow 0]{L^1}J_s^{-1}$ as well as $\left|\left|J_{s,h}^{-1}\right|\right|_{L^\infty}\leq C e^{\barre{r}_s}$ with $C$ a constant. Using the theorem \ref{existence} for the convergence of $\trho_{1,h}$ and $\trho_{2,h}$ gives the result.
\end{rem}

\begin{rem}
  In practical situations we are often only interested in the number of metastases and not in the density $\rho$ itself. Notice that thanks to the formula $\int_\Om \rho(t,X)dX=\int_0^t\int_{\d\Om}\trho_1(t,\tau,\sig)d\sig d\tau + \int_\Om \rho^0(X)dX$, we don't have to compute the jacobians $J_1$, $J_2$ to get the number of metastases. Yet, we still have to compute the characteristics since they are requested in the computation of the boundary condition (see formula \eqref{approxDonneesCont}).
\end{rem}

\section{Numerical simulations}

\subsection{Simulation technique and parameters}

Since our equation is two-dimensional, the computational cost can be relatively high because of the integral term in the boundary, especially for large-time simulations. In order to take into account that the metastases are born with a vasculature very close to a given value $\tt_0$, we examine replacing the function $N(\sigma)$ by a dirac measure. In \cite{benzekry2}, we demonstrate that if we take $N(1,\tt)=N^\eps(1,\tt)=\frac{1}{2\eps}\mathbf{1}_{\tt\in[\tt_0 - \eps,\,\tt_0+\eps]}$ and let $\eps$ go to zero the solution of the problem \eqref{equationNonAutonome}, in the case of an autonomous velocity field $G$ and initial condition equal to zero, converges to the measure solution of a limit problem consisting in replacing $N$ by a dirac measure in $(1,\tt_0)$. We use here these considerations to reduce the computational cost and simulate only along the characteristic passing through $(1,\tt_0)$, that is to say the scheme \eqref{schemaRho1} with only one discretization point $\sig_0$ on $\d\Om$ and $N(\sig)=\mathbf{1}_{\sig=\sig_0}$. Moreover, we use a Runge-Kutta method of order 4 for the approximation of the characteristics and a trapezoid method for the approximation of the boundary condition.
\\
The values of the parameters for the tumoral growth are taken from \cite{folkman}, where they were fitted to mice data. Following \cite{iwata} and \cite{BBHV}, we take $\alpha=2/3$ and fix the value of $m$ arbitrarily. The values of the parameters (without the treatment) are gathered in the table \ref{valParam}. The size (= volume) is expressed in $mm^3$ though until now it was thought as the number of cells.
\begin{table}[ht]
$\begin{array}{|c|c|c|c|c|c|c|}
\hline
a & c & d & x_0 \text{ (initial $x$)} & \theta_0\text{ (initial }\theta\text{)}& m & \alpha \\
(day^{-1}) & (day^{-1}) & (day^{-1}vol^{-2/3}) & (vol) & (vol) & (Nb\,of\,meta)(day^{-1})(vol^{-\alpha}) & \\
\hline
0.192 & 5.85 & 8.73\times 10^{-3} &10^{-6} &  625 & 10^{-3} & 2/3 \\
\hline
\end{array}$
\caption{Values of the parameters without treatment.}\label{valParam}
\end{table}

\subsection{Simulations without treatment}

A very important issue for clinicians is to determine the number of metastases which are not visible with medical imaging techniques (micro-metastases). Having a model for the density of metastases structured in size allows us to compute the number of visible and non-visible metastases. We took as threshold for a metastasis to be visible a size of $10^8$ cells, that is $100\;mm^3$ by using the conversion $1\;mm^3\simeq 10^6$ cells. In the figure \ref{VisibleVSTotal}, we plotted the result of a simulation showing both the total number of metastases as well as only the visible ones. We observe that at day $20$ the model predicts approximately one metastasis though it is not visible. At the end of the simulation, the total number of metastases is much bigger than the number of visible ones.\\
Thus, an interesting application of the model would be to help designing a predictive tool for the total number of metastases present in the organism of the patient. In this perspective, we define a metastatic index as the integral of $\rho$ on $\Om$ : 
$$MI(t):=\int_\Om \rho(t,X)dX.$$
Of course, this index depends on the values of the parameters, for example on the parameter $m$, as shown in the table \ref{variation_m}. The larger $m$, the larger the metastatic index. 
\begin{table}[!h]
$\begin{array}{|c|c|c|c|}
\hline
 & MI(1.5) & MI(7.5)    & MI(15)  \\
\hline
m=10^{-4}  & 5.80\times 10^{-3} & 6.60\times 10^{-2} & 2.79\times 10^{-1} \\ 
\hline
m=10^{-3}  & 5.80\times 10^{-2 }& 6.60\times 10^{-1} & 2.81 \\ 
\hline
m=10^{-2}  & 5.80\times 10^{-1} & 6.62 & 30.1 \\ 
\hline
\end{array}$
\caption{Variation of the number of metastases with respect to $m$.}\label{variation_m}
\end{table}
In this table, we remark that at least for times less than $15$ days, it seems that the metastatic index is linear in $m$. Indeed, this can be explained by the fact that at the beginning, most of the metastases come from the primary tumor and not by the metastases themselves (see figure \ref{primVSMeta}.A). This means that the renewal term in the boundary condition of \eqref{equationNonAutonome} could be neglected for small times and that the solution of \eqref{equationNonAutonome} is close to the one of
\edpn{\d_t \rho + \div(\rho G)=0 \\
-\Gn \rho(t,\sig)=N(\sig)\beta(X_p(t)) \\
\rho^0(X)=0 .
}
But then, integrating the equation on $\Om$ gives $MI(t)=\int_0^t \beta(X_p(s))ds=m\int_0^t x_p(s)^\alpha ds$, where $X_p(s)=(x_p(s),\tt_p(s))$ represents the primary tumor and solves the system \eqref{g1}-\eqref{g2} with initial condition $(x_0,\tt_0)$. The figure \ref{primVSMeta}.B shows that for larger times metastases emitted by the metastases themselves are more important than the ones emitted by the primary tumor. The metastatic index for large time is then not anymore linear in $m$ (result not shown).

\undessincm{VisibleVSTotal}{Evolution of the total number of metastases and of the number of visible metastases, that is whose size is bigger than $100\,mm^3(\simeq 10^8\,cells)$.}{VisibleVSTotal}{5}

\deuxdessinscm{primVSMetaT50}{primVSMetaT100}{Number of metastases emitted by the primary tumor and by the metastases themselves. A. T=50. B. T=100}{primVSMeta}{5}

\subsection{Simulations with treatment}

\subsubsection{Anti-angiogenic drug alone}

\paragraph{}We present various simulations of treatments, first involving an anti-angiogenic drug (AA) alone, in order to investigate the difference in effectiveness of various drugs regarding to their pharmacokinetic/pharmacodynamic parameters. The first result shown in figure \ref{DroguesFolkman} takes the three drugs which were used in \cite{folkman} where only the effect on tumor growth was investigated, and simulates the effect on the metastases.  The three drugs are TNP-470, endostatine and angiostatine and each drug is characterized by two parameters in the model : its efficacy $e$ and its clearance rate $clr_A$. These parameters were retrieved in \cite{folkman} by fitting the model to mice data. The administration protocols are the same for endostatine and angiostatine ($20$ mg every day) but for TNP-470 the drug is administered with a dose of $30$ mg every two days.  We observe that TNP-470 seems to have the poorest efficacy due to its large clearance. As noticed in \cite{folkman}, the ratio $e/clr_A$ should govern the efficacy of the drug and its value is $0.13$ for TNP-470 and $0.39$ for both endostatine and angiostatine. The model we developed is now able to simulate efficacy of the drugs on the metastatic evolution (figure \ref{DroguesFolkman}.C). Interestingly, the drug which seems to be more efficient regarding to the tumor size at the end of the simulation (day $15$), namely angiostatine, is not the one which gives the best result on the metastases. Indeed, the lower efficacy of endostatine regarding to ultimate size is due to a relatively high clearance provoking a quite fast rebound of the angiogenic capacity once the treatment stops. But since the tumor size was lower for longer time, the metastatic evolution was better contained. This shows that the model could be a helpful tool for the clinician since the response to a treatment can differ from the primary tumor to metastases, but the clinician has no data about micro-metastases which are not visible with imagery techniques.
\paragraph{}One of our main postulate in the treatment of cancer is that for a given drug, the effect can vary regarding to the temporal administration protocol of the drug, due to the combination of the pharmacokinetic of the drug and the intrinsic dynamic of tumoral and metastatic growth. To investigate the effect of varying the administration schedule of the drug, we simulated various administration protocols for the same drug (endostatine). The results are presented in figure \ref{Endostatine}. We gave the same dose and the same number of administrations of the drug but either uniformly distributed during 10 days (endostatine 2), concentrated in 5 days (endostatine 1) or in 2 days and a half (endostatine 3). We observe that the tumor is better stabilized with a uniform administration of the drug (endostatine 2) but the number of metastases is better reduced with the intermediate protocol (endostatine 1). It is interesting to notice that again if we look at the effects at the end of the simulation, the results are different for the tumor size and for the metastases. The two protocols endostatine 1 and endostatine 2 give the same size at the end, but not the same number of metastases. Moreover, the best protocol regarding to minimization of the final number of metastases (endostatine 1) is neither the one which provoked the largest regression of the tumor during the treatment (endostatine 3) nor the one with the most stable tumor dynamic (endostatine 2). In the figure \ref{EndostatineDose}, we investigate the influence of the AA dose (parameter $D_A$) on tumoral, vascular and metastatic evolution. We observe that the model is consistent since it exhibits a monotonous response to variation of the dose.

\begin{figure}[p]\begin{center}
    \includegraphics[width=6cm]{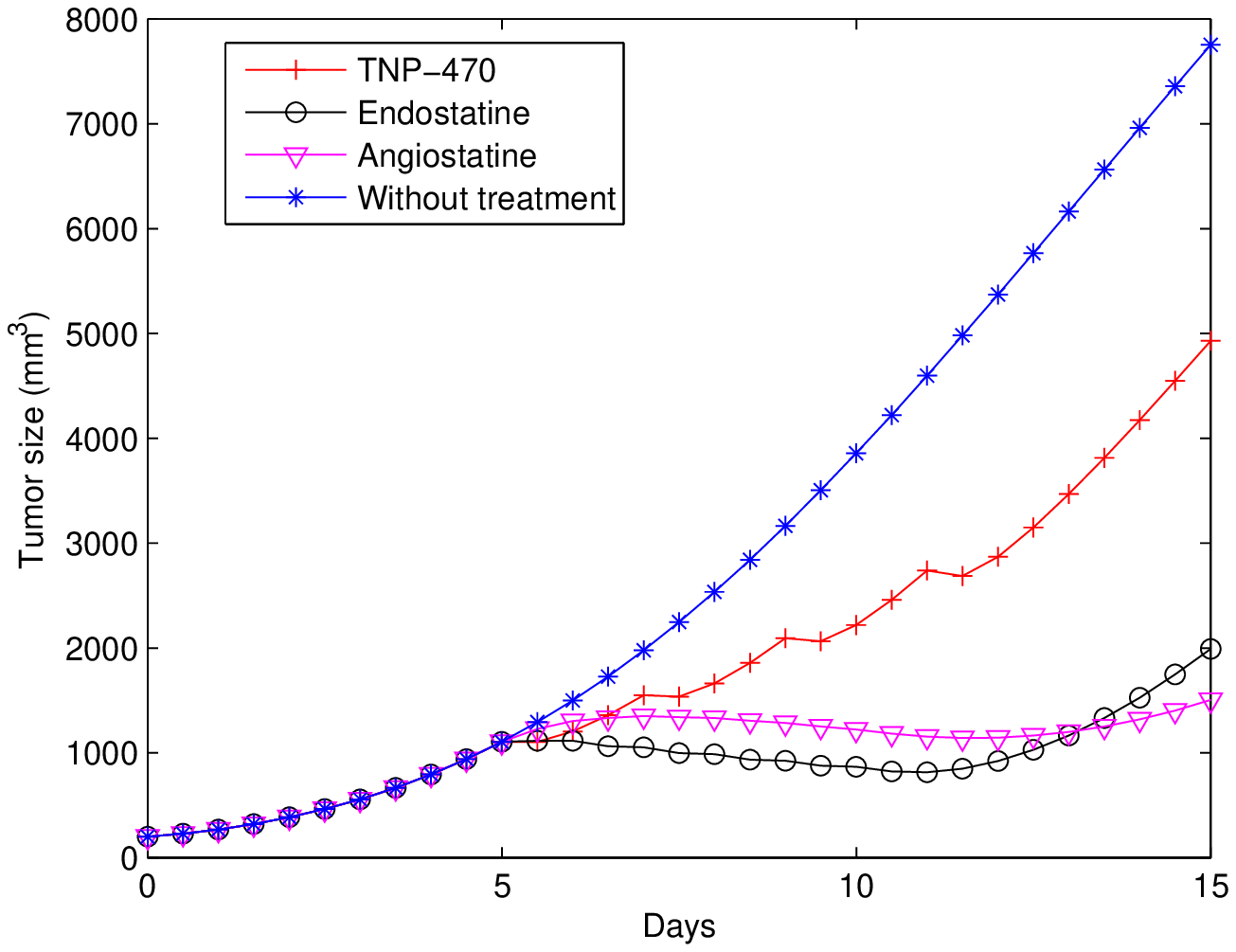} \rotatebox[origin=c]{0}{A}\hskip 0.3cm
\includegraphics[width=6cm]{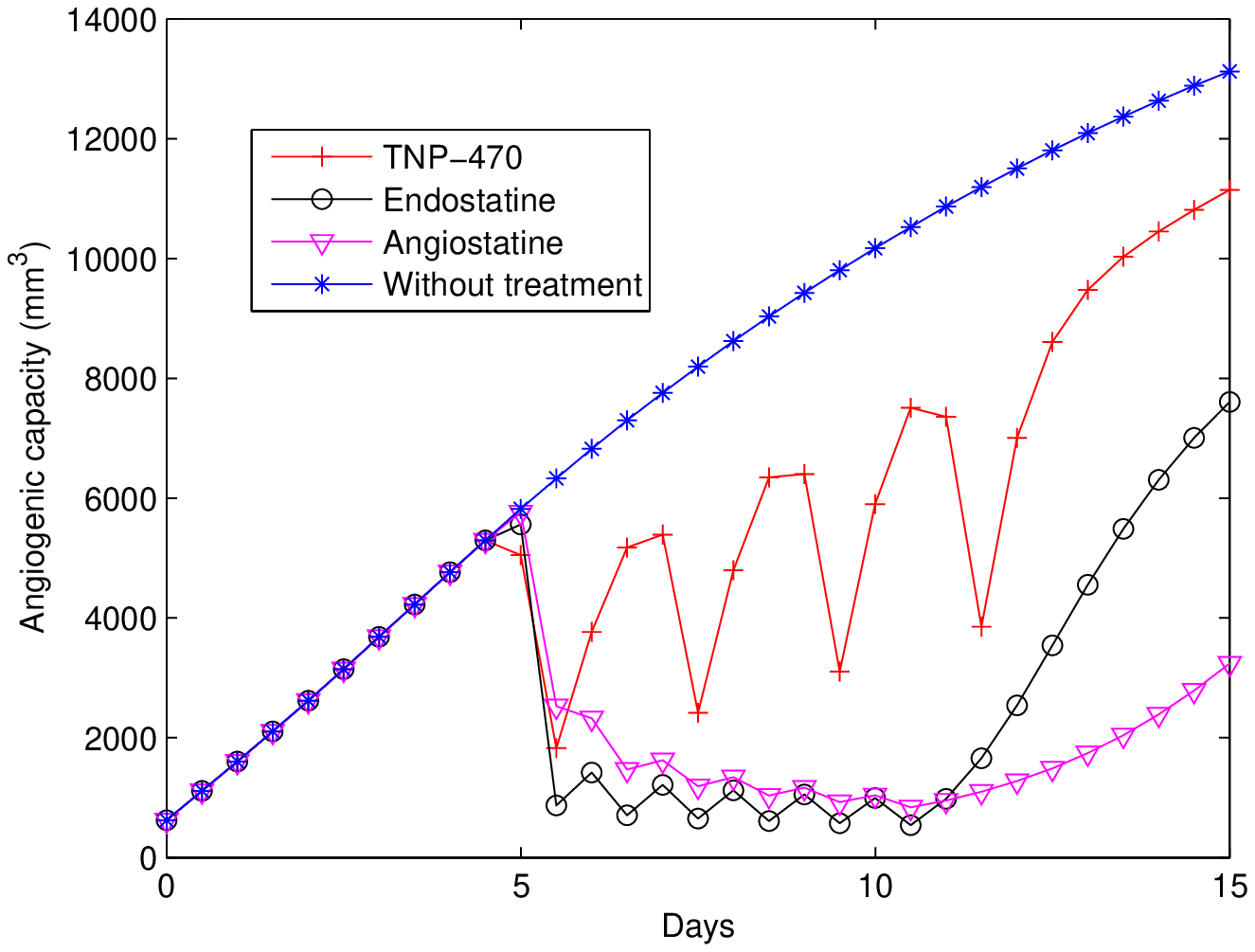} \rotatebox[origin=c]{0}{B}\hskip 0.3cm
\includegraphics[width=12cm]{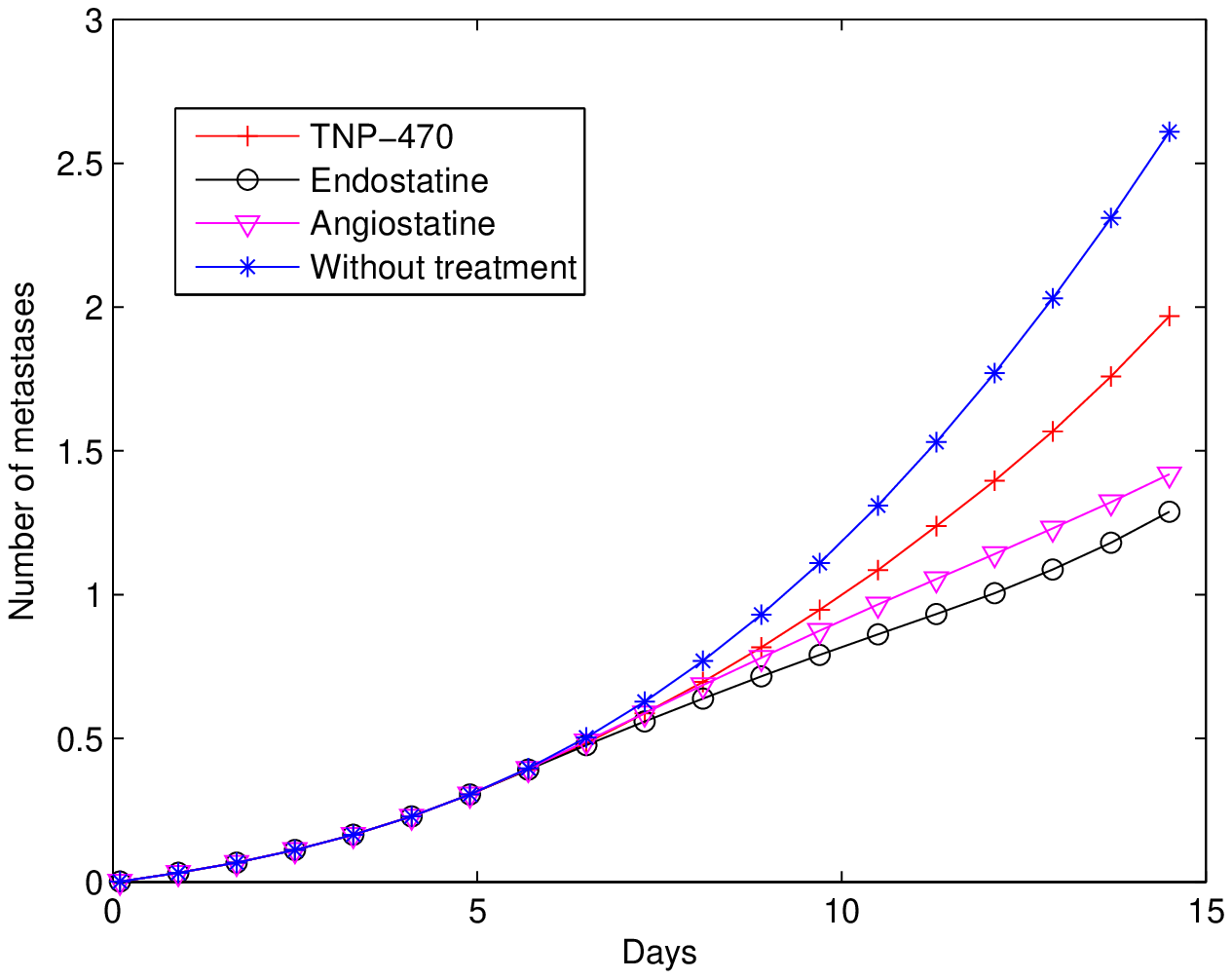} \rotatebox[origin=c]{0}{C}
    \end{center}
    \caption{Effect of the three drugs from \cite{folkman}. The treatment is administered from days 5 to 10. Endostatine ($e = 0.66, \, clr_A = 1.7 $) 20 mg every day, TNP-470 ($e = 1.3, \, clr_A = 10.1 $) 30 mg every two days and Angiostatine ($e = 0.15, \, clr_A = 0.38 $) 20 mg every day. A : Tumor size. B : Angiogenic capacity. C : Number of metastasis.}\label{DroguesFolkman}
\end{figure}

\begin{figure}[p]\begin{center}
    \includegraphics[width=6cm]{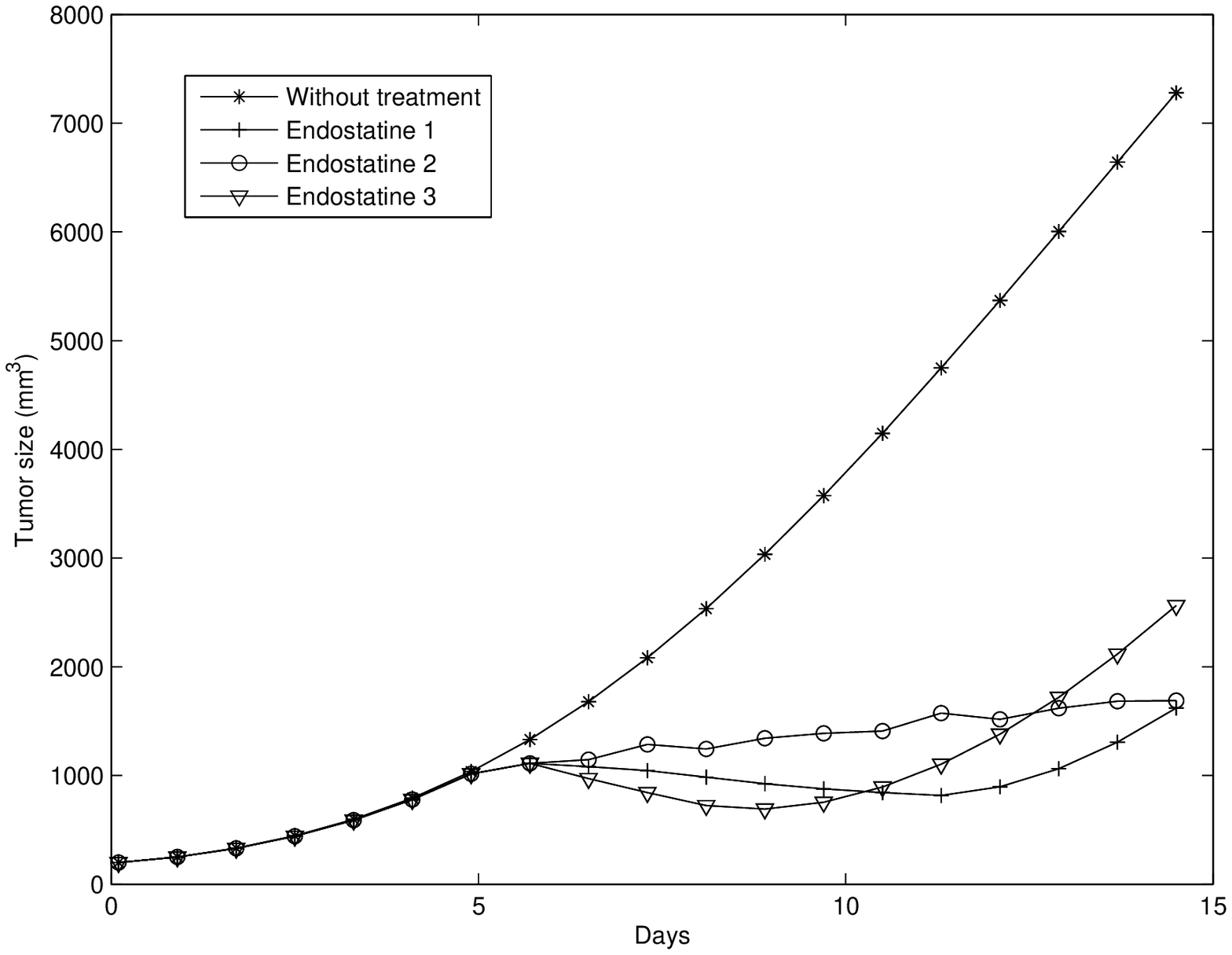} \rotatebox[origin=c]{0}{A}\hskip 0.3cm
\includegraphics[width=6cm]{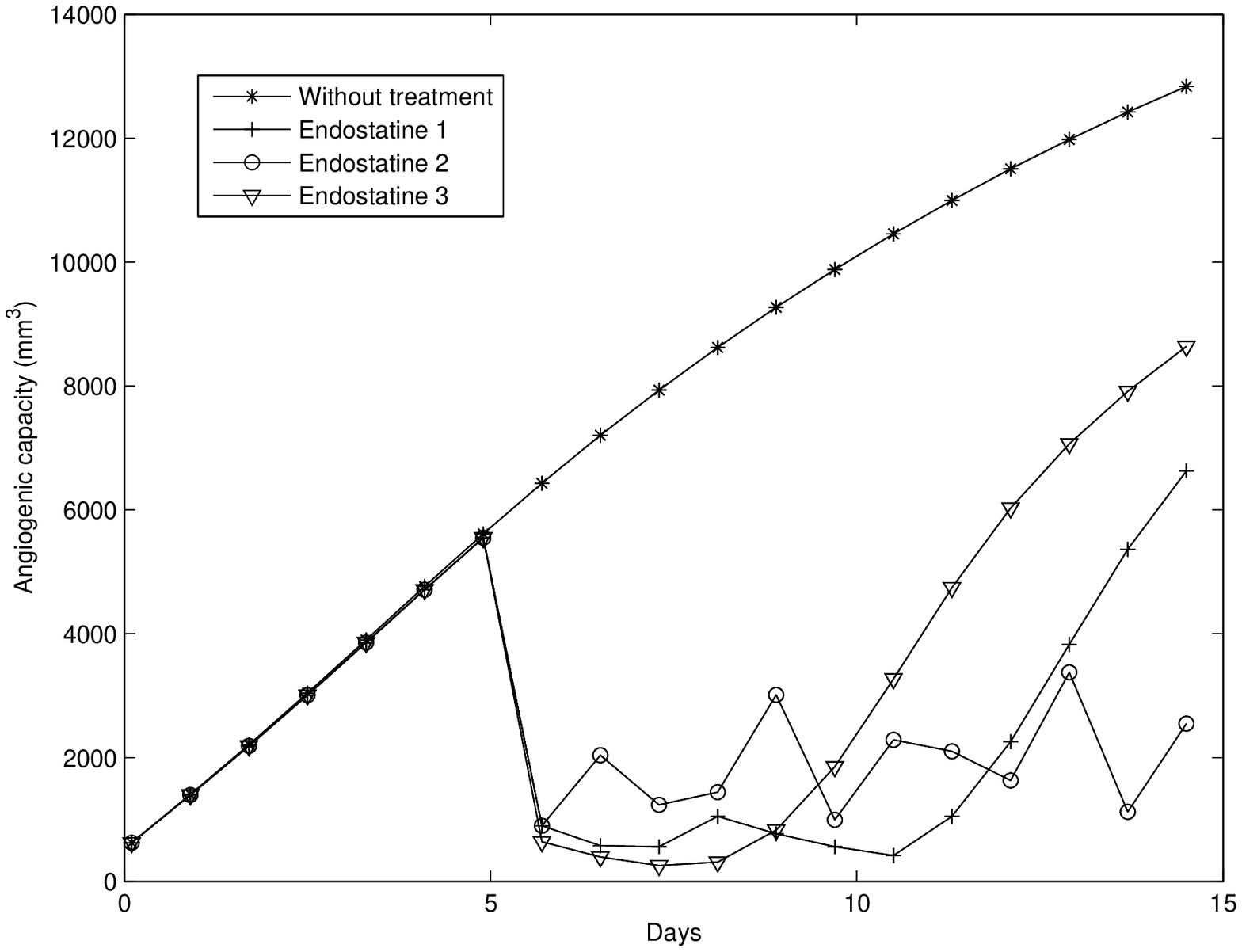} \rotatebox[origin=c]{0}{B}\hskip 0.3cm
\includegraphics[width=12cm]{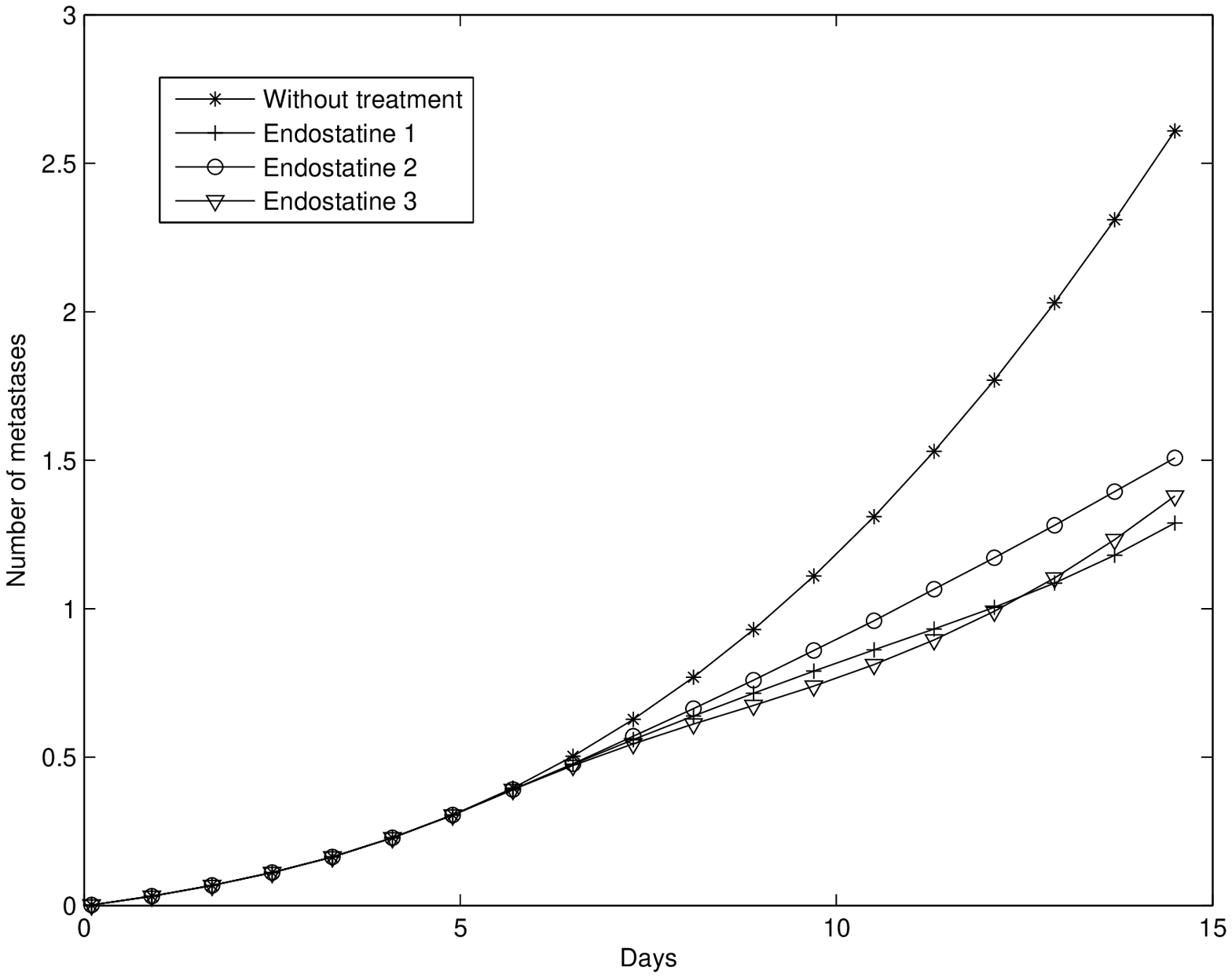} \rotatebox[origin=c]{0}{C}
    \end{center}
    \caption{Three different temporal administration protocols for the same drug (Endostatine). Same dose (20 mg) and number of administrations (6) but more or less concentrated at the beginning of the treatment. Endostatine 1 : each day from day 5 to 10. Endostatine 2 : every two days from day 5 to 15. Endostatine 3 : twice a day from day 5 to 7.5. A : Tumor size. B : Angiogenic capacity. C : Number of metastasis.}\label{Endostatine}
\end{figure}
    
\begin{figure}[p]\begin{center}
\includegraphics[width=6cm]{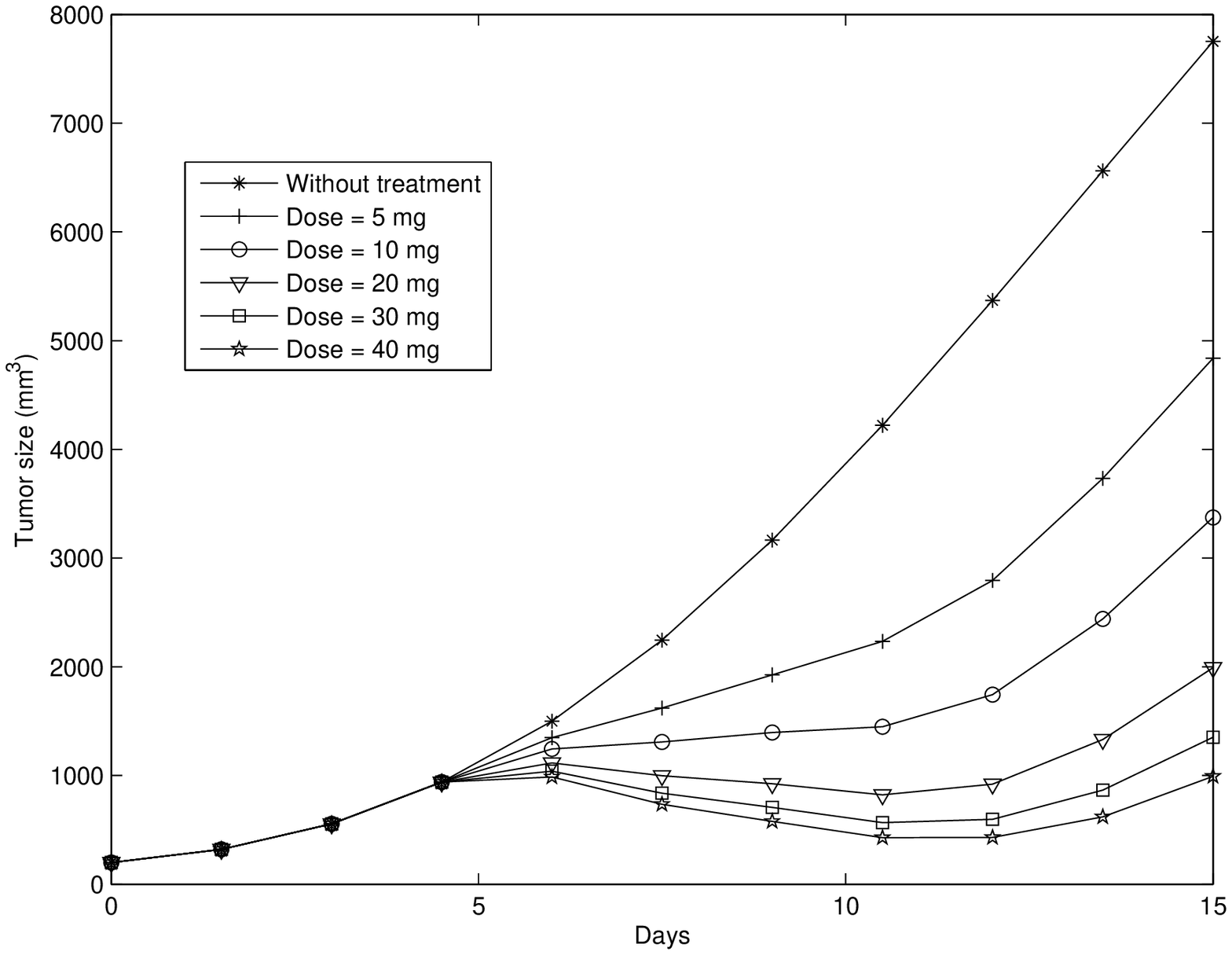} \rotatebox[origin=c]{0}{A}\hskip 0.3cm
\includegraphics[width=6cm]{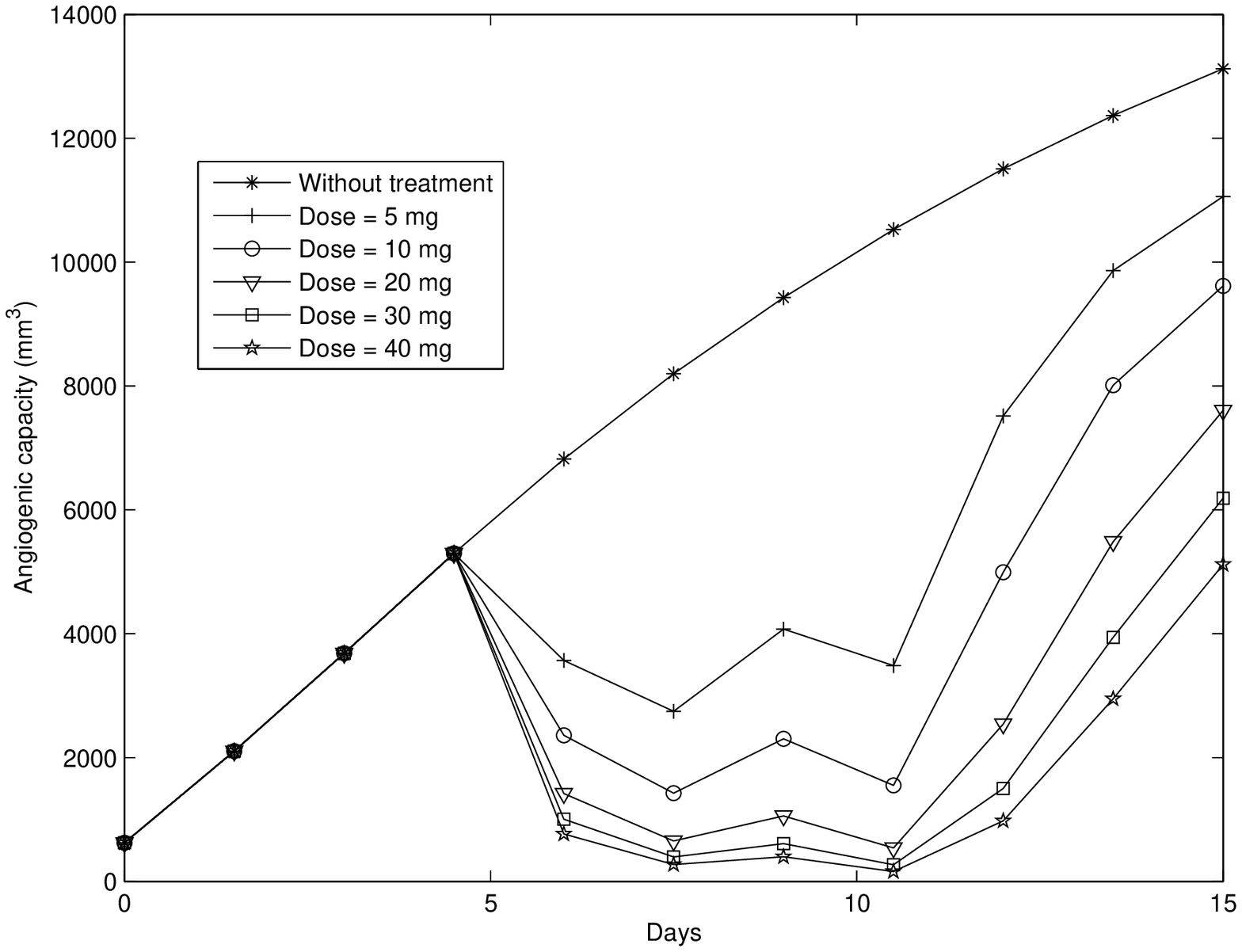} \rotatebox[origin=c]{0}{B}\hskip 0.3cm
\includegraphics[width=12cm]{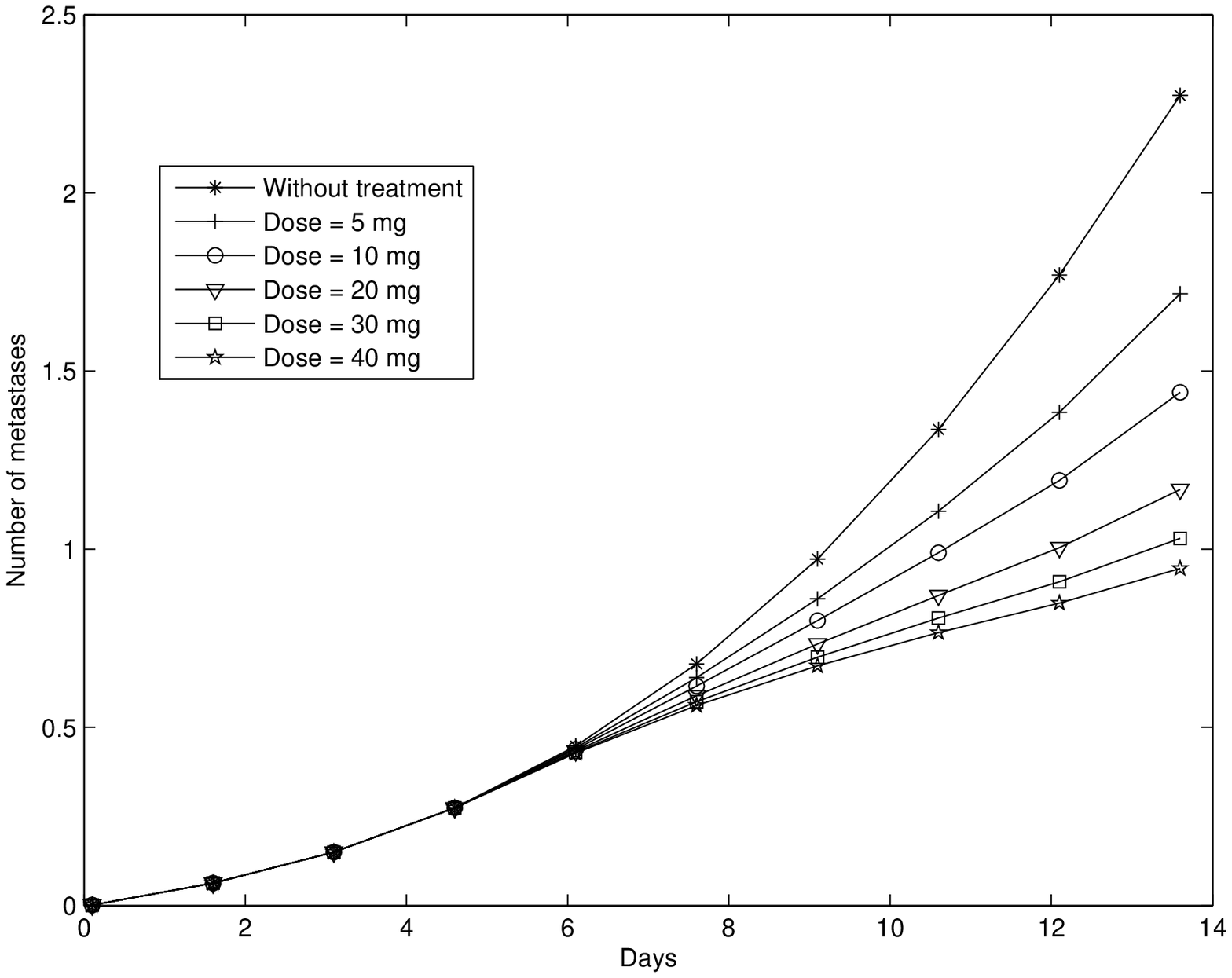} \rotatebox[origin=c]{0}{C}
    \end{center}
    \caption{Effect of the variation of the dose for endostatine. A : Tumor size. B : Angiogenic capacity. C : Number of metastasis.}\label{EndostatineDose}
\end{figure}

\subsubsection{Combination of anti-angiogenic and cytotoxic drug}

An important question in clinical oncology is to determinate how to combine a cytotoxic drug (CT) that kills the proliferative cells and an anti-angiogenic (AA) drug which acts on the angiogenic process, either by blocking the angiogenic factors like VEGF (monoclonal antibodies, e.g. Bevacizumab) or by inhibiting the receptors to this molecule. The AA drugs are classified as part of the cytostatic drugs as they aim to stabilize the disease. For instance, in the treatment of breast cancer, patients which express the receptor HER receive a combination of Docetaxel (CT) and Herceptine (a tyrosine kinase inhibitor, AA). Two questions are still open : which drug should come before the other and then what is the best temporal repartition for each drug? Here, we perform a brief \textit{in silico} study of the first question. Since we don't have real parameters for the cytotoxic drug we fix arbitrarily the value of each parameter $h$, $clr_C$ and $D_C$ to $1$, and perform simulations of the model to investigate combination of the CT and the AA. In the figure \ref{CombinaisonCTAA}, we present the results of two simulations : one giving the AA before the CT (fig. \ref{CombinaisonCTAA}.A) and the other one doing the opposite (fig. \ref{CombinaisonCTAA}.B). Although in both cases the effect on the metastases is very good since the growth seems stopped (fig. \ref{CombinaisonCTAA}.D), it appears that the qualitative behaviors of the tumoral and metastatic responses are different regarding to the order of administration of the drugs (fig. \ref{CombinaisonCTAA}.C and \ref{CombinaisonCTAA}.D). According to the model, it would be better to administrate first the CT in order to reduce the tumor burden and then use the AA to stabilize the disease. Indeed, the number of metastasis at the end of the simulation is lower when the CT is applied first than in the opposite case. Of course, this conclusion depends on the tumoral growth and drugs parameters but this simulation shows that the model is able to exhibit different responses regarding to the order of administration between CT and AA drugs. 

\quatredessins{TailleCA_AACT}{TailleCA_CTAA}{CTAA_tum}{CTAA_meta}{Combination of an anti-angiogenic drug (AA) : endostatine ($e=0.66$, $clr_{A}=1.7$ and $D_{A}=20 mg$) and a cytotoxic one (CT). The parameters for the CT are : $h=1$,  $clr_{C}=1$, and $D_{C}=1$. A.  AA from day 5 to 10 then CT from day 10 to 15, every day. Tumor growth and angiogenic capacity. B. CT from day 5 to 10 then AA from day 10 to 15, every day. Tumor growth and angiogenic capacity. C : Comparison between both combinations on the tumor growth. D : Comparison between both combinations on the metastatic evolution.}{CombinaisonCTAA}

\section{Conclusion}
In this paper, we combined the models of \cite{iwata} and \cite{folkman} to obtain a model aiming at describing the effect of anti-angiogenic drugs on the metastatic growth. We established the well-posedness of the model and developed an efficient numerical scheme to perform simulations, which could be adapted to similar models in higher dimensions. The model can now be used in order to rationalize the temporal administration of the anti-angiogenic drugs. To achieve this, we have to implement the various pharmacokinetic models of the different AA drugs and then compare the \textit{in silico} predictions to real patient data. \\
An important open problem in this direction is the mathematical parameter identifiability of the model, that is to say the inverse problem of uniqueness of the parameters resulting in a given observation. It is also important to develop efficient numerical methods able to achieve the parameter identification from the data. Indeed, identifying the parameters $m$ and $\alpha$ in a given patient could determine the metastatic aggressiveness of its cancer, through the metastatic index. This could lead to interesting clinical applications such as a refinement of the existing classifications like TNM or SBR, which deal only with the visible metastases. \\
As shown in \cite{ebos}, the metastatic response to AA treatment depends on the time schedule of the drug. The results of the simulations are encouraging in the perspective of using the model as a tool able to test various real temporal administration protocols of the drugs and to perform predictions of the mathematically optimized schedule for a given drug. Moreover, AA are never used in a monotherapy but rather combined with a cytotoxic drug, and determining the best way to combine both drugs is still a clinical open question \cite{riely}. As shown in the figure \ref{CombinaisonCTAA}, the model could help in this direction, regarding both to tumor regression and metastatic evolution of the disease. We should also develop further the modeling in order to take into account for the competition effects between CT and AA. Indeed, by reducing the vasculature AA drugs should induce worse supply of both drugs and on the contrary some arguments are expressed in favor of a normalization effect on the tumor vasculature by AA therapy \cite{jain}, at least at the beginning of the treatment. These elements should be incorporated to the model \textit{via} nonlinear terms involving the drugs concentrations in the equations \eqref{g1}-\eqref{g2}. The relative simplicity of the model ($6$ parameters without treatment) is a great advantage in view of concrete applications since we have to be able to fit the model to patients' data in order to retrieve their parameters and then perform predictions about the optimal schedule.\\
A fundamental problem that we have to integrate in our model is the one of toxicities which have to be dynamically controlled to optimize the scheduling of the drug. In the case of CT and on the tumoral growth, a model dealing with hematological toxicities is used to drive phase I clinical trials \cite{barboTrial, barbo}. In our case, we also have to integrate a module to control the toxicity and address the resulting problem of optimization under constraints.\\
Eventually, our model can be used to run \textit{in silico} tests about the paradigm of metronomic chemotherapies which consists in delivering a cytotoxic drug at low doses and uniformly distributed in the treatment cycle rather than administrating the maximum tolerate dose (MTD) at the beginning of the cycle. Indeed, these metronomic protocols seem to have a dynamical anti-angiogenic effect \cite{hahnfeldt03, barboMetro} that can be integrated in the model of \cite{folkman} for the tumour growth and in our model for the effect on metastases.

\appendix

\section{Proof of the proposition \ref{propChgtVar}}

 The result for the second map is classical. For the first one, we have to deal with irregular points of the boundary $\d\Om$. We denote by $\chi$ the set of such points and set $\chi_t:=\{X(t;\tau,\xi);\;\xi \in \chi,\;0\leq \tau \leq t\}$. In order to prove the result, it is sufficient to prove that the map 
$$X_1^t:\begin{array}{ccc} ]0,t[\times \d\Om\setminus\chi & \rightarrow & \Om_1^t\setminus\chi_t \\
		       (\tau,\sig)                  & \mapsto      & X(t;\tau,\sig)
  \end{array}$$
is a diffeomorphism, that globally the map $X_1^t:[0,t]\times\d\Om\rightarrow \bOm_1^t$ is bilipschitz and that its inverse is $X\mapsto(\tau^t(X),\sig^t(X))$. For the first point, since we avoid the irregular points of the boundary by excluding the set $\chi$, we have the $C^1$ regularity. It remains to prove that $X_1^t(\tau,\sig)$ is one-to-one and onto, and that its inverse is $C^1$.\\
\espace $\bullet$\textit{ The map $X_1^t$ is one-to-one and onto}. Let $t>0$ and $X\in \Om_1^t$. We have $X_1^t(\tau^t(X),\sig^t(X))=X(t;\tau^t(X),\sig^t(X))=X(t;\tau^t(X),X(\tau^t(X);t,X))=X(t;t,X)=X$.\\
 For the injectivity, we remark that if we have $X(t;\tau,\sig)=X(t;\tau',\sig')$ with for instance $\tau'<\tau$, then $\sig=X(\tau;\tau',\sig')$ which is prohibited by the assumption that $\Gn(\tau,\sig)>0$. Thus $X_1^t$ is one-to-one and we have, for $(\tau,\sig)\in [0,t]\times \d\Om$ : $X(t;\tau^t(X_1^t(\tau,\sig)),\sig(X_1^t(\tau,\sig)))=X(t;\tau,\sig)$ which implies $\tau^t(X_1^t(\tau,\sig))=\tau$. Thus, we have proven that the inverse of $X_1^t$ is $X\mapsto(\tau^t(X),\sig^t(X))$. \\
\espace $\bullet$\textit{ The map $X_1^t$ is a diffeomorphism}. We will prove the formula \eqref{formule_jacobien} for $J_1$ which will conclude the proof by using the local inversion theorem. We have $J_1(t;\tau,\sig)=|\d_\tau X_1^t \wedge \d_\sig X_1^t | $, with $\d_\sig X_1^t:= D_Y X \circ \sig'$ for $\sig$ being a parametrization of $\d\Om$ and $D_Y X \in \mathcal{M}_2(\R)$ the derivative in $Y$ of $X(t;\tau,Y)$ viewed as the flow on $\barre{\Om}$. We compute
\begin{align*}
\d_t (\d_\tau X_1^t \wedge \d_\sig X_1^t) & =\d_\tau \d_t X_1^t \wedge \d_\sig X_1^t + \d_\tau X_1^t \wedge \d_t (D_YX_1^t \circ \sig')
          = \d_\tau (G\circ X_1^t)\wedge \d_\sig X_1^t  + \d_\tau X_1^t \wedge DG\circ D_YX_1^t \circ \sig' \\
         & = DG\circ \d_\tau X_1^t \wedge \d_\sig X_1^t  + \d_\tau X_1^t \wedge DG\circ \d_\sig X_1^t 
          = \div(G) (\d_\tau X_1^t \wedge \d_\sig X_1^t).
\end{align*}
We compute now directly the value of $J_1(t;t,\sig)$. We define
$$T(h)=\frac{X_1^t(t;t+h,\sig) - X_1^t(t;t,\sig)}{h}$$
and now notice that we can write 
\begin{align*}
X_1^t(t;t,\sig) & =X_1^t(t;t+h,X_1^t(t+h;t,\sig))\\
      & =X_1^t(t;t+h,\sig)+D_Y X_1^t(t;t+h,\sig)(X_1^t(t+h;t,\sig) - X_1^t(t;t,\sig))+o(h)\\
      & = X_1^t(t;t+h,\sig) +h D_Y X_1^t(t;t+h,\sig)\circ G(t,\sig) + o(h).
\end{align*}
Now when $h$ goes to zero $D_Y X_1^t(t;t+h,\sig) \rightarrow D_Y X_1^t(t;t,\sig)=Id$ since $X_1^t(t;t,Y)=Y$. Finally, we have $T(h) \rightarrow -G(t,\sig)$, thus $\d_\tau X_1^t(t;t,\sig)=-G(t,\sig)$ and $\d_\tau X_1^t \wedge \d_\sig X_1^t (t;t,\sig)=-G(t,\sig)\wedge \sig'=G(t,\sig)\cdot\nu(\sig)$. Solving the differential equation between times $\tau$ and $t$ and taking the absolute value then gives the formula \eqref{formule_jacobien}.\\
\espace $\bullet$\textit{ Globally, $X_1^t$ is bilipschitz.} It is possible to show that $|||DX_1^t|||_{L^\infty([0,t]\times\d\Om)}\leq e^{t|||DG|||_{L^\infty([0,T]\times\bOm}}$. On the other hand, using the formula $(DX_1^t)^{-1}=J_1^{-1} ~^tCom(DX_1^t)$ and the fact that from \eqref{formule_jacobien} $J_1^{-1}$ is bounded on $\bOm_1^t$ thanks to the assumption \eqref{hypG} we have $|||(DX_1^t)^{-1}|||_{L^\infty(\bOm_1^t)}<\infty$. Thus $X_1^t$ and $(X_1^t)^{-1}$ are Lipschitz on $[0,t]\times\d\Om\setminus\chi$ and $\Om_1^t\setminus\chi_t$ respectively, and they are both globally continuous on $[0,t]\times\d\Om$ and $\bOm_1^t$. Hence they are globally Lipschitz.

\begin{rem}
 Using the same technique than in the previous proof, we can calculate the
derivative of $X_1(t;\tau,\sig)$ in the $\tau$ direction. Indeed we
compute, for all $t,\tau,\sig$
\begin{align*}
 X_1(t;\tau,\sig)&=X_1(t;\tau +h, X_1(\tau +h;\tau,\sig)) \\
		 &=X_1(t;\tau+h,\sig) + D_Y
X_1(t;\tau+h,\sig)(X_1(\tau +h;\tau,\sig)-X_1(\tau;\tau,\sig) ) +
o(h) \\
		  &=X_1(t;\tau+h,\sig) +hD_Y X_1(t;\tau+h,\sig)\circ
G(\tau,\sig) + o(h)
\end{align*}
which gives 
\begin{equation}\label{derivee_tau_Phi}
\d_\tau
X_1(t;\tau,\sig)=\underset{h\rightarrow 0}{\rm
lim}\frac{X_1(t;\tau+h,\sig) - X_1(t;\tau,\sig)}{h}=- D_Y
X_1(t;\tau,\sig)\circ G(\tau,\sig).  
\end{equation}
\end{rem}

\bibliographystyle{plain}

\subsection*{Acknowledgment}
The author would like to express its gratitude to the following people for great support and helpful discussions : D. Barbolosi, A. Benabdallah, F. Hubert and F. Boyer. This work was partially supported by ANR project MEMOREX.

\end{document}